\documentclass[11pt]{amsart}
\usepackage[a4paper, margin=1in]{geometry}
\usepackage{amsfonts}
\usepackage{amssymb}
\usepackage{amsmath}
\usepackage{amsthm}
\usepackage{mathtools}
\usepackage{hyperref}
\usepackage{cleveref}
\usepackage{stmaryrd}
\usepackage{tikz-cd}
\usepackage{enumitem}
\usepackage{multirow}
\usepackage{bbm}
\usepackage{comment}
\usepackage{float}
\usepackage{url}
\usepackage[en-GB]{datetime2}
\DTMlangsetup[en-GB]{ord=omit}
\allowdisplaybreaks

\title{The complexity of Ford domains of $\Gamma_0(N)$}
\author{Pengcheng Zhang}
\address{Max Planck Institute for Mathematics, Vivatsgasse 7, 53111 Bonn, Germany}
\email{pzhang@mpim-bonn.mpg.de}
\date{\today}

\newtheorem{theorem}{Theorem}[section]

\newtheorem{lemma}[theorem]{Lemma}
\newtheorem{proposition}[theorem]{Proposition}
\newtheorem{corollary}[theorem]{Corollary}

\theoremstyle{remark}
\newtheorem{remark}[theorem]{Remark}

\def\MR#1{}

\newcommand\sm[1]{\begin{psmallmatrix}#1\end{psmallmatrix}}
\newcommand\cl[1]{\mathcal{#1}}

\newcommand{\Z}{\mathbb{Z}}

\newcommand{\C}{\mathbb{C}}
\newcommand{\clh}{\mathcal{H}}
\newcommand{\eq}{\;=\;}
\newcommand{\deq}{\;:=\;}
\def\thin{{\hskip 1pt}}
\newcommand{\hash}{\texttt{\#}}
\newcommand{\divides}{\thin|\thin}

\renewcommand{\Re}{\operatorname{Re}}

\begin{document}

\begin{abstract}
We investigate a particular choice of the Ford fundamental domain of the congruence subgroup $\Gamma_0(N)$ and define a notion of complexity $c(N)$ accordingly, which is a nonnegative integer and carries some information on the shape of the Ford domain. The property that $c(N)=0$ first appeared as a technical assumption in a paper by Pohl, which is closely related to a conjecture of Zagier on the ``reduction theory'' of $\Gamma_0(N)$. In this paper, we give a complete classification of positive integers $N$ with $c(N)=0$, and we also show that $c(N)$ goes to infinity if both the number of distinct prime factors of $N$ and the smallest prime factor of $N$ go to infinity.
\end{abstract}

\maketitle

\setcounter{tocdepth}{1}
\tableofcontents

\section{Introduction}
\label{intro}

Let $N$ be a positive integer. The object of interest to us is the region
\begin{align*}
    \cl{R}_N\deq \bigcup_{\substack{N\divides b,\; \gcd(a,b)=1, \\ 0\leq \frac{a}{b}\leq 1}}\bigg\{z\in\C\;\bigg|\; \big|z-\tfrac{a}{b}\big|\leq \tfrac{1}{b}\bigg\}.
\end{align*}
The consideration of this region is motivated by the construction of Ford fundamental domains, which we briefly recall here.

For a Fuchsian group $\Gamma$ with a cusp at $\infty$, one can construct a Ford fundamental domain for $\Gamma$ in the following way. One chooses a fundamental domain $\cl{F}_\infty\subseteq\clh$ for the stabilizer group $\Gamma_\infty$ of $\infty$ and considers
\begin{align*}
    \cl{R}_\Gamma\deq\bigcup_{\sm{a&b\\c&d}\in\Gamma\thin\setminus\thin\Gamma_\infty}\bigg\{z\in\C\;\bigg|\; \big|z+\tfrac{d}{c}|\leq \tfrac{1}{c}\bigg\}.
\end{align*}
Then,
\begin{align*}
    \cl{F}_\Gamma\deq \cl{F}_\infty \setminus (\cl{R}_\Gamma\cap\clh)
\end{align*}
is (roughly) a Ford fundamental domain of the Fuchsian group $\Gamma$, where $\clh$ denotes the upper half plane. For a detailed discussion of this construction, we refer to \cite[Section~3.3]{katok-fuchsian} or \cite[Section~2.2]{pohl14}.

From this point of view, the region $\cl{R}_N$ that we are interested in should correspond to the case when $\Gamma$ is the congruence subgroup $\Gamma_0(N)$. In this case, $\Gamma_\infty=\langle\sm{1&1\\0&1}\rangle$ and one common choice of $\cl{F}_\infty$ is given by $\{x+iy\mid 0\leq x\leq 1, y>0\}$.\footnote{Technically $\cl{R}_{\Gamma_0(N)}\neq\cl{R}_N$, since the definition of $\cl{R}_N$ already incorporates the choice of $\cl{F}_\infty$ by requiring that $0\leq\frac{a}{b}\leq 1$.} 

To rewrite $\cl{R}_N$ in a cleaner way, let
\begin{align*}
    D(\alpha,r)\deq\big\{z\in\C\;\big|\;|z-\alpha|\leq r\big\}
\end{align*}
be the closed disk centered at $\alpha\in\C$ with radius $r\geq 0$, and for any rational number $\frac{a}{b}\in[0,1]$ with $\gcd(a,b)=1$, let
\begin{align*}
    D_{\frac{a}{b}}\deq D\big(\tfrac{a}{b},\tfrac{1}{b}\big).
\end{align*}
Then,
\begin{align*}
    \cl{R}_N\eq\bigcup_{\substack{N\divides b,\; \gcd(a,b)=1, \\ 0\leq \frac{a}{b}\leq 1}}D_{\frac{a}{b}}.
\end{align*}
In fact, there exists a unique finite set $S_N$ consisting of rational numbers $\frac{a}{b}\in[0,1]$ with $N\divides b$ and \mbox{$\gcd(a,b)=1$} satisfying that
\begin{enumerate}[label=$(\arabic*)$]
    \item $\cl{R}_N=\bigcup_{\alpha\in S_N}D_\alpha$;
    \item $S_N$ is minimal with respect to the first property.
\end{enumerate}
Visually, $S_N$ consists of those elements $\alpha$ satisfying that an arc of positive length of the boundary of $D_\alpha$ lies on the boundary of $\cl{R}_N$.

For each $\alpha=\frac{a}{b}\in S_N$, define the \emph{north pole} of $\alpha$ as the point $P_\alpha:=\frac{a}{b}+\frac{1}{b}\cdot i$ on the boundary of $D_\alpha$. For $\alpha\in S_N$, define the \emph{complexity} of the north pole $P_\alpha$ as
\begin{align*}
    c(P_\alpha)\deq\hash\{\beta\in S_N\mid \beta\neq\alpha\text{ and } P_\alpha\in D_\beta\},
\end{align*}
i.e., the number of disks that lie strictly on the boundary of $\cl{R}_N$ and cover $P_\alpha$. Define the \emph{complexity} of the integer $N$ as
\begin{align*}
    c(N)\deq\max_{\alpha\in S_N}c(P_\alpha).
\end{align*}

The motivation of this paper for studying the complexity $c(N)$ originates from a conjecture of Don Zagier \cite{zagier} on the ``reduction theory'' of $\Gamma_0(N)$,\footnote{We would also like to refer to \cite{katok-ugarcovici-theory-2010,katok-ugarcovici-structure-2010,katok-ugarcovici-2017,abrams-katok-ugarcovici} for detailed discussions of Zagier's conjecture.} which itself was motivated by reconstructing modular forms from period polynomials. It was mentioned in \cite[p.~8]{zagier} that the methods developed by Anke Pohl and Paul Wabnitz \cite{pohl14,pohl-wabnitz,wabnitz} should settle Zagier's conjecture in the cases where a technical assumption called Condition (A) (see \cite[p.~2181]{pohl14}) is satisfied. In the setting of $\Gamma_0(N)$, this assumption is equivalent to the property that $c(N)=0$, i.e., every north pole $P_\alpha$ of $\alpha\in S_N$ is not covered by any other disk that lies strictly on the boundary of~$\cl{R}_N$. In the thesis of Nicolas Herzog, a former bachelor student of Anke Pohl, it was conjectured that
\begin{center}
    $c(N)>0$ $\Leftrightarrow$
    there exist distinct odd primes $p_1,p_2,p_3$ such that $2p_1p_2p_3\divides N$.
\end{center}
This is in general not true, but we will see how the parity and the number of distinct prime factors of $N$ play a role in determining the complexity $c(N)$. In particular, we will classify all the integers $N$ with $c(N)=0$ completely, which should then imply Zagier's conjecture for those particular integers~$N$.

Throughout this paper, we adopt the following notations. Let $\omega(N)$ denote the number of distinct prime factors of $N$, and let $p_i(N)$ denote the $i$-th smallest prime factor of~$N$ for $1\leq i\leq \omega(N)$. We will also simply write $p_i=p_i(N)$ when there is no confusion.

The first result of this paper is a complete classification of integers $N$ with $c(N)=0$.

\begin{theorem}
\label{mainthm}
Let $N$ be a positive integer. Then, $c(N)=0$ if and only if one of the following holds
\begin{enumerate}[label=\rm{(\arabic*)}]
    \item $\omega(N)\leq 3$, i.e., $N$ has at most three distinct prime factors;
    \item $\omega(N)=4$ and $p_1(N)\geq 3$, i.e., $N$ is odd and has exactly four distinct prime factors.
\end{enumerate}
\end{theorem}

The second result of this paper gives us some idea of how $c(N)$ grows with respect to $N$. It roughly says that $c(N)$ is large if both the number of distinct prime factors of $N$ and the smallest prime factor of $N$ are large.

\begin{theorem}
\label{compinfty}
Let $N$ be a positive integer and let $q$ be a fixed prime number. If
\begin{align*}
    \omega(N)\;\geq\;\frac{q^6}{2} + 3q^4 - \frac{3q^3}{2} + \frac{9q^2}{2} - \frac{11q}{2} + 3
\end{align*}
and
\begin{align*}
    p_1(N)\;\geq\;\frac{q^6}{2} + 3q^4 - 2q^3 + \frac{9q^2}{2} - 7q + 3,
\end{align*}
then $c(N)\geq q-2$. In particular, the complexity function $c(N)$ is not bounded.
\end{theorem}

Let us briefly discuss the case when $N=1$, or equivalently, $\omega(N)=0$. In this case, it is easy to see that $\cl{R}_1=D_0\cup D_1$, so $c(1)=0$. As we will see later, the case when $N=1$ is in some sense different from the case when $N>1$, since it is the only case where the disks $D_0$ and $D_1$ show up in the region $\cl{R}_N$. Hence, one can already observe the ``influence'' of $\omega(N)$ even in the simplest case. Due to the different behavior, we will carry the assumption that $N>1$ from now on.

To finish this section, we give a brief outline of how the analysis of $c(N)$ and $\cl{R}_N$ goes. For $N=p$ a prime, it is easy to see that $\cl{R}_p=\cup_{1\leq a\leq p-1}D_{\frac{a}{p}}$ and hence $c(p)=0$. Our overall strategy for understanding $\cl{R}_N$ can then be described as dividing it into good regions and bad regions, where the good regions behave like $\cl{R}_p$ and are easy to describe while the behaviors of the bad regions depend on the prime factors of $N$ and thus need to be treated more carefully.

\subsection*{Acknowledgements}
The author would like to thank Anke Pohl and Don Zagier for introducing this topic to the author as well as for various useful discussions on this topic. The author would also like to thank the anonymous reviewer for their detailed suggestions.

\section{Preparatory lemmas}
\label{regionsect}

We will prove some preparatory lemmas for general $N>1$ in this section. These lemmas essentially help us divide $\cl{R}_N$ into good and bad regions as discussed previously, and give a simple description of the good regions.

\begin{lemma}
\label{notcoprime}
For each integer $n$ with $0\leq n\leq N$ and $\gcd(N,n)>1$, the point $z=\frac{n}{N}$ lies on the boundary of $\cl{R}_N$.
\end{lemma}
\begin{proof}
As the statement is vacuously true if $N=1$, we assume that $N>1$. First, suppose that $n=0$. By definitions, $0\in D_{\frac{1}{N}}\subseteq\cl{R}_N$, so $0\in\cl{R}_N$. Now, suppose that $n\geq 1$. Let $d=\gcd(N,n)>1$ and consider $\frac{a}{b}$ with
\begin{align*}
    a\eq\frac{nN}{d}-1\quad\text{ and }\quad b\eq\frac{N^2}{d}.
\end{align*}
Indeed, $N\divides b$. Note that $\gcd(\frac{nN}{d}-1,N)=1=\gcd(\frac{nN}{d}-1,\frac{N}{d})$, so $\gcd(a,b)=1$. Also, $1\leq n\leq N$ so $0\leq\frac{a}{b}\leq 1$. Now, 
\begin{align*}
    \frac{n}{N}-\frac{a}{b}\eq\frac{n\cdot\frac{N}{d}}{N\cdot\frac{N}{d}}-\frac{n\cdot\frac{N}{d}-1}{N\cdot\frac{N}{d}}\eq\frac{1}{b},
\end{align*}
so $\frac{n}{N}\in\cl{R}_N$.

To prove that $\frac{n}{N}$ lies on the boundary, suppose that there exists some $\frac{a'}{b'}$ with $0\leq\frac{a'}{b'}\leq 1$, $\gcd(a',b')=1$, and $N\divides b'$ such that $\big|\frac{n}{N}-\frac{a'}{b'}\big|<\frac{1}{b'}$. Write $b'=b_0N$ for some positive integer~$b_0$. Then, $\big|\frac{b_0n-a'}{b_0N}\big|<\frac{1}{b_0N}$. This implies that $b_0n=a'$, which contradicts that $\gcd(N,a')=1$ as $\gcd(N,n)>1$.
\end{proof}
\begin{remark}
In fact, points of the form $\frac{n}{N}$ with $0\leq n\leq N$ and $\gcd(N,n)>1$ are the only real points on the boundary of $\cl{R}_N$, but we will not need this for our proof.
\end{remark}

Let $\psi(N)=N-\phi(N)$ denote the number of integers $n$ with $1\leq n\leq N$ and $\gcd(N,n)>1$, where $\phi(N)$ denotes Euler's totient function. For $N>1$, we enumerate all the integers $n$ with $0\leq n\leq N$ and $\gcd(N,n)>1$ as $0=n_0<n_1<\cdots<n_{\psi(N)-1}<n_{\psi(N)}=N$. Define
\begin{align}
\label{region}
    \cl{R}_{N,n_l}\deq\bigg\{z\in\cl{R}_N\;\bigg|\; \frac{n_l}{N}\leq\Re(z)\leq\frac{n_{l+1}}{N}\bigg\}
\end{align}
for each $0\leq l\leq\psi(N)-1$. By \Cref{notcoprime},
\begin{align*}
    \cl{R}_{N,n_{l-1}}\cap\cl{R}_{N,n_l}\eq\bigg\{\frac{n_l}{N}\bigg\}
\end{align*}
for each $0\leq l\leq\psi(N)-1$. Hence, the region $\cl{R}_N$ can be divided into the regions $\cl{R}_{N,n_l}$ by ``disconnecting'' $\cl{R}_N$ at the points $\{\frac{n_l}{N}\}$ (see the visualization of $\cl{R}_9$ in \Cref{visual-R9}). In particular, each $\cl{R}_{N,n_l}$ can also be written as
\begin{align*}
   \cl{R}_{N,n_l}\eq\bigcup_{\substack{N\divides b,\; \gcd(a,b)=1, \\ \frac{n_l}{N}\leq \frac{a}{b}\leq \frac{n_{l+1}}{N}}}D_{\frac{a}{b}},
\end{align*}
and there exists a unique minimal finite set $S_{N,n_l}\subseteq[\frac{n_l}{N},\frac{n_{l+1}}{N}]$ of rational numbers (compared with $S_N$) such that 
\begin{align*}
    \cl{R}_{N,n_l}\eq\bigcup_{\alpha\in S_{N,n_l}}D_\alpha.
\end{align*}
Note that for each $n_l$, the integer $n_{l+1}$, being the next integer after $n_l$ with \mbox{$\gcd(N,n_{l+1})>1$}, is completely determined by $n_l$. We will thus drop the index $l$ whenever convenient.

Below is a visualization of $\cl{R}_N$ when $N=9$. The region $\cl{R}_9$ can be naturally divided into three regions $\cl{R}_{9,0}$, $\cl{R}_{9,3}$, and $\cl{R}_{9,6}$ by ``disconnecting'' it at red points, which are exactly of the form $\frac{n}{N}$ with $\gcd(N,n)>1$.

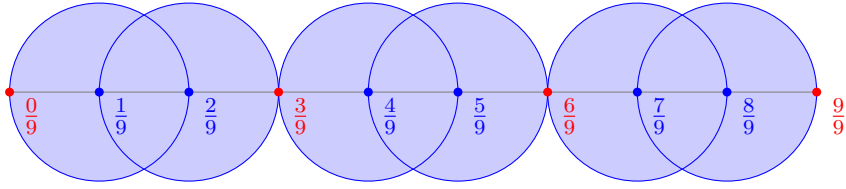
\begin{figure}[H]
\centering
\begin{tikzpicture}[scale=1.18]
\fill[fill=blue!20!white] (1,0) circle [radius=1cm];
\fill[fill=blue!20!white] (2,0) circle [radius=1cm];
\fill[fill=blue!20!white] (4,0) circle [radius=1cm];
\fill[fill=blue!20!white] (5,0) circle [radius=1cm];
\fill[fill=blue!20!white] (7,0) circle [radius=1cm];
\fill[fill=blue!20!white] (8,0) circle [radius=1cm];
\draw[gray, thin](0,0)--(9,0);
\draw[blue] (1,0) circle [radius=1cm];
\draw[blue] (2,0) circle [radius=1cm];
\draw[blue] (4,0) circle [radius=1cm];
\draw[blue] (5,0) circle [radius=1cm];
\draw[blue] (7,0) circle [radius=1cm];
\draw[blue] (8,0) circle [radius=1cm];
\draw (0,0) node[below=9pt, right=1pt, red]{$\frac{0}{9}$};
\draw (1,0) node[below=9pt, right=1pt, blue]{$\frac{1}{9}$};
\draw (2,0) node[below=9pt, right=1pt, blue]{$\frac{2}{9}$};
\draw (3,0) node[below=9pt, right=1pt, red]{$\frac{3}{9}$};
\draw (4,0) node[below=9pt, right=1pt, blue]{$\frac{4}{9}$};
\draw (5,0) node[below=9pt, right=1pt, blue]{$\frac{5}{9}$};
\draw (6,0) node[below=9pt, right=1pt, red]{$\frac{6}{9}$};
\draw (7,0) node[below=9pt, right=1pt, blue]{$\frac{7}{9}$};
\draw (8,0) node[below=9pt, right=1pt, blue]{$\frac{8}{9}$};
\draw (9,0) node[below=9pt, right=1pt, red]{$\frac{9}{9}$};
\fill[red] (0,0) circle [radius=0.05cm];
\fill[blue] (1,0) circle [radius=0.05cm];
\fill[blue] (2,0) circle [radius=0.05cm];
\fill[red] (3,0) circle [radius=0.05cm];
\fill[blue] (4,0) circle [radius=0.05cm];
\fill[blue] (5,0) circle [radius=0.05cm];
\fill[red] (6,0) circle [radius=0.05cm];
\fill[blue] (7,0) circle [radius=0.05cm];
\fill[blue] (8,0) circle [radius=0.05cm];
\fill[red] (9,0) circle [radius=0.05cm];
\end{tikzpicture}
\caption{\label{visual-R9}Visualization of $\cl{R}_9$}
\end{figure}

We would also like to define a notion of complexity for these regions~$\cl{R}_{N,n_l}$, or specifically, for the pair $(N,n)$. From now on, whenever we say a \emph{pair} $(N,n)$, we always mean a pair of integers $(N,n)$ with $0\leq n\leq N-1$ and $\gcd(N,n)>1$ unless stated otherwise. Given a pair $(N,n)$, define the \emph{complexity} of the pair $(N,n)$ as
\begin{align*}
    c(N,n)\deq\max_{\alpha\in S_{N,n}}c(P_\alpha).
\end{align*}
We have the straightforward lemma.

\begin{lemma}
\label{paircomp}
Let $N>1$ be an integer. Then,
\begin{align*}
    c(N)\eq\max_{\substack{0\leq n\leq N-1 \\ \gcd(N,n)>1}}c(N,n).
\end{align*}
In particular, $c(N)=0$ if and only if $c(N,n)=0$ for all pairs $(N,n)$.
\end{lemma}

To formulate our next lemma, we make a pair of definitions. The pair $(N,n)$ is called \emph{good} if $\gcd(N,n+1)=1$, and \emph{bad} if $\gcd(N,n+1)>1$. We also call the corresponding region $\cl{R}_{N,n}$ \emph{good} or \emph{bad} if the pair $(N,n)$ is good or bad respectively.

\begin{lemma}
\label{goodregion}
Any good region $\cl{R}_{N,n_l}$ is of the form
\begin{align*}
    \cl{R}_{N,n_l}\eq\bigcup_{n_l<a<n_{l+1}}D_{\frac{a}{N}}.
\end{align*}
In particular, $c(N,n)=0$ for all good pairs $(N,n)$.
\end{lemma}
\begin{proof}
Let $\cl{R}=\bigcup_{n_l<a<n_{l+1}}D_{\frac{a}{N}}$. It is clear from the definition that $\gcd(N,a)=1$ for $n_l<a<n_{l+1}$, so $\cl{R}_{N,n_l}\supseteq\cl{R}$. For the reverse inclusion, consider any $D_\frac{a}{b}$ with $N\divides b$, $\gcd(a,b)=1$, and $\frac{n_l}{N}\leq \frac{a}{b}\leq \frac{n_{l+1}}{N}$. If $b=N$, then indeed $D_{\frac{a}{b}}\subseteq\cl{R}$. Suppose that $b\geq 2N$. Then, $D_\frac{a}{b}$ is contained in the region
\begin{align*}
    \bigcup_{\frac{2n_l+1}{2N}\leq\alpha\leq\frac{2n_{l+1}-1}{2N}}D\big(\alpha,\tfrac{1}{2N}\big)
\end{align*}
where $\alpha$ ranges over all real numbers in the interval. Visually, this is the region swept by a closed disk from $D\big(\frac{2n_l+1}{2N},\frac{1}{2N}\big)$ to $D\big(\frac{2n_{l+1}-1}{2N},\frac{1}{2N}\big)$ (see the green region in \Cref{visual-proof-good-region}). Indeed, this swept region is contained in~$\cl{R}$ since $\frac{1}{2N}<\frac{\sqrt{3}}{2N}$, and hence $D_\frac{a}{b}\subseteq\cl{R}$.

\begin{figure}[H]
\centering
\begin{tikzpicture}[scale=1.18]
\fill[fill=green!20!white,draw=green] (0.5,0.5) arc [start angle=90, end angle=270, radius=0.5cm];
\fill[fill=green!20!white] (0.49,-0.5) rectangle (8.51,0.5);
\fill[fill=green!20!white,draw=green] (8.5,-0.5) arc [start angle=-90, end angle=90, radius=0.5cm];
\draw[green](0.49,0.5)--(8.51,0.5);
\draw[green](0.49,-0.5)--(8.51,-0.5);
\draw[gray, thin](0,0)--(9,0);
\draw[dotted](4.2,0.7)--(4.8,0.7);
\draw[dotted](4.2,-0.7)--(4.8,-0.7);
\draw (1,0) circle [radius=1cm];
\draw (2,0) circle [radius=1cm];
\draw (3,0) circle [radius=1cm];
\draw (4,1) arc [start angle=90, end angle=270, radius=1cm];
\draw (5,-1) arc [start angle=-90, end angle=90, radius=1cm];
\draw (6,0) circle [radius=1cm];
\draw (7,0) circle [radius=1cm];
\draw (8,0) circle [radius=1cm];
\draw (0,0) node[below=9pt, left=1pt]{$\frac{n_l}{N}$};
\draw (9,0) node[below=9pt, right=1pt]{$\frac{n_{l+1}}{N}$};
\end{tikzpicture}
\caption{\label{visual-proof-good-region}Visualization of the proof}
\end{figure}
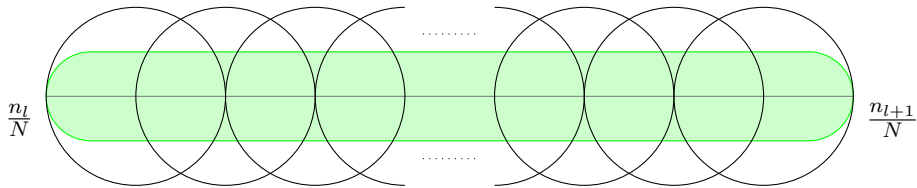
\end{proof}

\begin{corollary}
\label{badregion}
The complexity of $N$ only depends on the complexity of the bad pairs $(N,n)$, i.e., $c(N)=0$ if all the pairs $(N,n)$ are good, and
\begin{align*}
    c(N)\eq\max_{(N,n)\text{ bad}}c(N,n)\eq\max_{\substack{0\leq n\leq N-1 \\ \gcd(N,n)>1 \\ \gcd(N,n+1)>1}}c(N,n)
\end{align*}
otherwise. In particular, $c(N)=0$ if and only if $c(N,n)=0$ for all bad pairs $(N,n)$.
\end{corollary}

At this point one can already see how the number of distinct prime factors of $N$ plays a role in the behavior of $c(N)$. If $\omega(N)=1$, i.e., $N$ is a prime power, then all the pairs $(N,n)$ are good, and hence $c(N)=0$.

As a result of \Cref{badregion}, we will almost always focus on the bad pairs $(N,n)$ and hence the bad regions $\cl{R}_{N,n}$ from now on. In fact, we only need to focus on half of the bad pairs by the following lemma.

\begin{lemma}
\label{symmetric}
Let $(N,n)$ be a bad pair. Then, $(N,N-n-1)$ is also a bad pair, and
\begin{align*}
    c(N,n)\eq c(N,N-n-1).
\end{align*}
\end{lemma}
\begin{proof}
It is easy to see from the definition that $(N,N-n-1)$ is a bad pair. Note that for a rational number $\frac{a}{b}$ with $\gcd(a,b)=1$,
\begin{align*}
    D\big(\tfrac{a}{b},\tfrac{1}{b}\big)\subseteq \cl{R}_{N,n}&\Longleftrightarrow N\divides b,\;\gcd(a,b)=1,\;\tfrac{n}{N}\leq\tfrac{a}{b}\leq\tfrac{n+1}{N} \\
    &\Longleftrightarrow N\divides b,\;\gcd(b-a,b)=1,\;\tfrac{N-n-1}{N}\leq\tfrac{b-a}{b}\leq\tfrac{N-n}{N} \\
    &\Longleftrightarrow D\big(\tfrac{b-a}{b},\tfrac{1}{b}\big)\subseteq\cl{R}_{N,N-n-1}.
\end{align*}
Also, $D\big(\frac{a}{b},\frac{1}{b}\big)$ and $D\big(\frac{b-a}{b},\frac{1}{b}\big)$ are symmetric with respect to the line $\Re(z)=1/2$. This then means that $\cl{R}_{N,n}$ and $\cl{R}_{N,N-n-1}$ are also symmetric with respect to the line $\Re(z)=1/2$. In particular, as the complexity only depends on the shape of the regions $\cl{R}_{N,n}$ and $\cl{R}_{N,N-n-1}$, we have $c(N,n)=c(N,N-n-1)$.
\end{proof}

\section{The case when $\omega(N)\leq 2$}
\label{omega2-section}

In this section, we will show that $c(N)=0$ when $\omega(N)\leq 2$. We first prove a lemma on bad regions, and the result essentially yields from this lemma. The main idea behind this lemma is that disks of radius $\frac{1}{N}$ cannot appear in the bad region, so we are forced to look at disks of radius $\frac{1}{2N}$. As before, a pair $(N,n)$ always means a pair of integers $(N,n)$ with $0\leq n\leq N-1$ and $\gcd(N,n)>1$.

\begin{lemma}
\label{midpoint}
Let $(N,n)$ be a bad pair and suppose that $\gcd(N,2n+1)=1$. Then, 
\begin{align*}
    \cl{R}_{N,n}\eq D_{\frac{2n+1}{2N}}.
\end{align*}
In particular, $c(N,n)=0$.
\end{lemma}
\begin{proof}
It is clear that $\cl{R}_{N,n}\supseteq D_{\frac{2n+1}{2N}}$. The reverse inclusion is also clear since any disk in $\cl{R}_{N,n}$ cannot go past the endpoints $\frac{n}{N}$ and $\frac{n+1}{N}$ by \Cref{notcoprime} and hence must stay inside $D_{\frac{2n+1}{2N}}$.
\end{proof}

Below is a visualization of the case when $N=15$. The good regions are colored in blue while the bad regions are colored in red. The two bad regions are exactly of the form as described in \Cref{midpoint}.

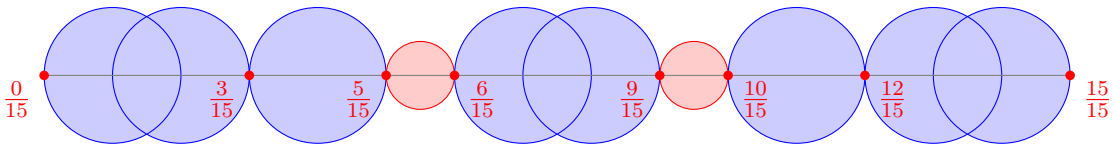
\begin{figure}[H]
\centering
\begin{tikzpicture}[scale=0.9]
\fill[fill=blue!20!white] (1,0) circle [radius=1cm];
\fill[fill=blue!20!white] (2,0) circle [radius=1cm];
\fill[fill=blue!20!white] (4,0) circle [radius=1cm];
\fill[fill=blue!20!white] (7,0) circle [radius=1cm];
\fill[fill=blue!20!white] (8,0) circle [radius=1cm];
\fill[fill=blue!20!white] (11,0) circle [radius=1cm];
\fill[fill=blue!20!white] (13,0) circle [radius=1cm];
\fill[fill=blue!20!white] (14,0) circle [radius=1cm];
\fill[fill=red!20!white] (5.5,0) circle [radius=0.5cm];
\fill[fill=red!20!white] (9.5,0) circle [radius=0.5cm];
\draw[blue] (1,0) circle [radius=1cm];
\draw[blue] (2,0) circle [radius=1cm];
\draw[blue] (4,0) circle [radius=1cm];
\draw[red] (5.5,0) circle [radius=0.5cm];
\draw[blue] (7,0) circle [radius=1cm];
\draw[blue] (8,0) circle [radius=1cm];
\draw[red] (9.5,0) circle [radius=0.5cm];
\draw[blue] (11,0) circle [radius=1cm];
\draw[blue] (13,0) circle [radius=1cm];
\draw[blue] (14,0) circle [radius=1cm];
\draw (0,0) node[below=9pt, left=1pt, red]{$\frac{0}{15}$};
\draw (3,0) node[below=9pt, left=1pt, red]{$\frac{3}{15}$};
\draw (5,0) node[below=9pt, left=1pt, red]{$\frac{5}{15}$};
\draw (6,0) node[below=9pt, right=1pt, red]{$\frac{6}{15}$};
\draw (9,0) node[below=9pt, left=1pt, red]{$\frac{9}{15}$};
\draw (10,0) node[below=9pt, right=1pt, red]{$\frac{10}{15}$};
\draw (12,0) node[below=9pt, right=1pt, red]{$\frac{12}{15}$};
\draw (15,0) node[below=9pt, right=1pt, red]{$\frac{15}{15}$};
\draw[gray, thin](0,0)--(15,0);
\fill[red] (0,0) circle [radius=0.07cm];
\fill[red] (3,0) circle [radius=0.07cm];
\fill[red] (5,0) circle [radius=0.07cm];
\fill[red] (6,0) circle [radius=0.07cm];
\fill[red] (9,0) circle [radius=0.07cm];
\fill[red] (10,0) circle [radius=0.07cm];
\fill[red] (12,0) circle [radius=0.07cm];
\fill[red] (15,0) circle [radius=0.07cm];
\end{tikzpicture}
\caption{Visualization of $\cl{R}_{15}$}
\end{figure}

\begin{corollary}
\label{2coprime}
For any pair $(N,n)$, $c(N,n)=0$ if one of
\begin{align*}
    n+1,\;2n+1
\end{align*}
is coprime to $N$.
\end{corollary}
\begin{proof}
Fix a pair $(N,n)$. If \mbox{$\gcd(N,n+1)=1$}, then $(N,n)$ is a good pair, so $c(N,n)=0$ by \Cref{goodregion}. If \mbox{$\gcd(N,n+1)>1$}, then $(N,n)$ is a bad pair and $\gcd(N,2n+1)=1$ by the assumption, so $c(N,n)=0$ by \Cref{midpoint}.
\end{proof}

\begin{proposition}
\label{propomega2}
Let $N>1$ be an integer. Then, $c(N)=0$ if for any integer $n$ with $0\leq n\leq N-1$, one of
\begin{align*}
    n,\;n+1,\;2n+1
\end{align*}
is coprime to $N$. 
\end{proposition}
\begin{proof}
By \Cref{paircomp}, it suffices to show that $c(N,n)=0$ for all the pairs $(N,n)$. Now, for a fixed pair $(N,n)$, the assumption implies that one of $n+1$ and $2n+1$ is coprime to $N$, so $c(N,n)=0$ by \Cref{2coprime}. The result then follows.
\end{proof}

\begin{corollary}
Let $N>1$ be an integer. If $\omega(N)\leq 2$, then $c(N)=0$.
\end{corollary}
\begin{proof}
For any $0\leq n\leq N-1$, the three integers $n,n+1$, and $2n+1$ are pairwise coprime, so $N$ cannot have common prime factor with all three of them. Hence, one of $n,n+1,2n+1$ is coprime to $N$, so $c(N)=0$ by \Cref{propomega2}.   
\end{proof}

The two key ingredients that we use here are \Cref{propomega2}, which gives a characterization of $c(N)$ via the coprimality of $N$ to three particular integers, and that the three particular integers are pairwise coprime, which forces $N$ with $\omega(N)\leq 2$ to be coprime to one of them. In the analysis of the remaining cases, we will also make use of the assumption on $\omega(N)$ (and $p_1(N)$) in this way. That is, we will first characterize $c(N)$ via the coprimality of $N$ to certain integers, and then show that those integers are pairwise coprime (or can only have small common prime divisors), which then forces $N$ with prescribed assumptions on $\omega(N)$ (and $p_1(N)$) to have certain complexity $c(N)$.

\section{The case when $\omega(N)=3$}
\label{omega3-section}

In this section, we will show that $c(N)=0$ when $\omega(N)=3$. This section is divided into two parts, $p_1(N)\geq 3$ and $p_1(N)=2$, i.e., $N$ is odd and $N$ is even. The odd case is more straightforward while the even case is more subtle.

\subsection{$\omega(N)=3$ and $p_1(N)\geq 3$}

As before, we first prove a lemma on bad regions. This lemma can be viewed as a further investigation following \Cref{midpoint} in the sense that we now have to discuss disks of radius $\frac{1}{3N}$.

\begin{lemma}
\label{threepoint}
Let $(N,n)$ be a bad pair. Suppose that $\gcd(N,2n+1)>1$ and 
$\gcd(N,3n+1)=\gcd(N,3n+2)=1$. Then, 
\begin{align*}
    \cl{R}_{N,n}\eq D_{\frac{3n+1}{3N}}\cup D_{\frac{3n+2}{3N}}.
\end{align*}
In particular, $c(N,n)=0$.
\end{lemma}
\begin{proof}
Write $\cl{R}=D_{\frac{3n+1}{3N}}\cup D_{\frac{3n+2}{3N}}$. It is clear that $\cl{R}_{N,n}\supseteq \cl{R}$. For the reverse 
inclusion, consider any $D_\frac{a}{b}$ with $N\divides b$, $\gcd(a,b)=1$, and $\frac{n}{N}\leq \frac{a}{b}\leq \frac{n+1}{N}$. We must have $b\geq 3N$ by the assumption and it is clear that $D_\frac{a}{b}\subseteq\cl{R}$ for $b=3N$.

Suppose that $b\geq 4N$. Then, $D_\frac{a}{b}$ is contained in the region
\begin{align*}
    \bigcup_{\frac{4n+1}{4N}\leq\alpha\leq\frac{4n+3}{4N}}D\big(\alpha,\tfrac{1}{4N}\big)
\end{align*}
where $\alpha$ ranges over all real numbers in the interval. Visually, this is the region swept by a closed disk from $D\big(\frac{4n+1}{4N},\frac{1}{4N}\big)$ to $D\big(\frac{4n+3}{4N},\frac{1}{4N}\big)$ (see the green region in \Cref{visual-proof-3n}). Indeed, this swept region is contained in $\cl{R}$ since $\frac{1}{4N}<\frac{\sqrt{3}}{2}\cdot\frac{1}{3N}$, and hence $D_\frac{a}{b}\subseteq\cl{R}$.

\begin{figure}[H]
\centering
\begin{tikzpicture}[scale=0.439]
\fill[fill=green!20!white,draw=green] (3,3) arc [start angle=90, end angle=270, radius=3cm];
\fill[fill=green!20!white] (2.99,-3) rectangle (9.01,3);
\fill[fill=green!20!white,draw=green] (9,-3) arc [start angle=-90, end angle=90, radius=3cm];
\draw[green](2.99,3)--(9.01,3);
\draw[green](2.99,-3)--(9.01,-3);
\draw[gray, thin](0,0)--(12,0);
\draw (4,0) circle [radius=4cm];
\draw (8,0) circle [radius=4cm];
\draw (0,0) node[below=9pt, left=1pt]{$\frac{n}{N}$};
\draw (4,0) node[below=9pt, left=1pt]{$\frac{3n+1}{3N}$};
\draw (8,0) node[below=9pt, right=1pt]{$\frac{3n+2}{3N}$};
\draw (12,0) node[below=9pt, right=1pt]{$\frac{n+1}{N}$};
\fill[black] (0,0) circle [radius=0.13cm];
\fill[black] (4,0) circle [radius=0.13cm];
\fill[black] (8,0) circle [radius=0.13cm];
\fill[black] (12,0) circle [radius=0.13cm];
\end{tikzpicture}
\caption{\label{visual-proof-3n}Visualization of the proof}
\end{figure}
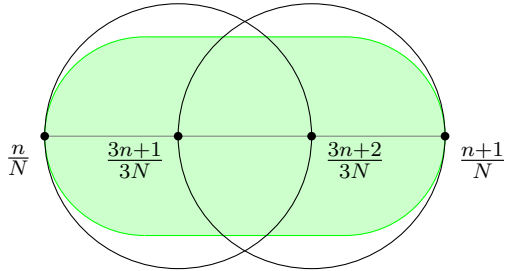
\end{proof}

\begin{corollary}
\label{3coprime}
For any pair $(N,n)$, $c(N,n)=0$ if two of
\begin{align*}
    n+1,\;2n+1,\;3n+1,\;3n+2
\end{align*}
are coprime to $N$. 
\end{corollary}
\begin{proof}
As $c(N,n)=0$ by \Cref{goodregion} if $(N,n)$ is a good pair, we will assume that $(N,n)$ is a bad pair, so that $\gcd(N,n+1)>1$. If $\gcd(N,2n+1)=1$, then $c(N,n)=0$ by \Cref{midpoint}. If \mbox{$\gcd(N,2n+1)>1$}, then $\gcd(N,3n+1)=\gcd(N,3n+2)=1$ by the assumption, so $c(N,n)=0$ by \Cref{threepoint}.
\end{proof}

\begin{proposition}
\label{propomega3}
Let $N>1$ be an integer. Then, $c(N)=0$ if for any integer $n$ with $0\leq n\leq N-1$, two of
\begin{align*}
    n,\;n+1,\;2n+1,\;3n+1,\;3n+2
\end{align*}
are coprime to $N$. 
\end{proposition}
\begin{proof}
By \Cref{paircomp}, it suffices to show that $c(N,n)=0$ for all the pairs $(N,n)$. Now, for a fixed pair $(N,n)$, the assumption implies that two of $n+1,2n+1,3n+1,3n+2$ are coprime to $N$, so $c(N,n)=0$ by \Cref{3coprime}. The result then follows.
\end{proof}

We have now obtained the first ingredient of the proof. The other ingredient of the proof on the common prime divisors is given by the following easy lemma, which also explains the reason for putting the extra restriction $p_1(N)\geq 3$ in this case.

\begin{lemma}
\label{pleq2}
Let $p$ be a prime and $n$ be an integer such that there exist distinct
\begin{align*}
    a,b\in\{n,n+1,2n+1,3n+1,3n+2\}
\end{align*}
with $p\divides a$ and $p\divides b$. Then, $p=2$.
\end{lemma}

\begin{corollary}
Let $N>1$ be an integer. If $\omega(N)=3$ and $p_1(N)\geq 3$, then $c(N)=0$.
\end{corollary}
\begin{proof}
Fix an integer $n$ with $0\leq n\leq N-1$. As $p_i\geq p_1\geq 3$, each $p_i$ divides at most one of $n,n+1,2n+1,3n+1,3n+2$ by \Cref{pleq2}. As $\omega(N)=3$, at most three of $n,n+1,2n+1,3n+1,3n+2$ are not coprime to $N$, so at least two of them are coprime to~$N$. Hence, $c(N)=0$ by \Cref{propomega3}.
\end{proof}

\subsection{$\omega(N)=3$ and $p_1(N)=2$}
We will still start by first showing two general lemmas on bad regions, which will be helpful to the later part of our proof as well. In these two lemmas, we now turn to disks of radius $\frac{1}{4N}$.

\begin{lemma}
\label{fourpoint1}
Let $(N,n)$ be a bad pair. Suppose that
\begin{enumerate}[label=\rm{(\alph*)}]
    \item $\gcd(N,2n+1)>1$;
    \item $\gcd(N,3n+1)=1$ and $\gcd(N,3n+2)>1$;
    \item $\gcd(N,4n+3)=1$.
\end{enumerate}
Then,
\begin{align*}
    \cl{R}_{N,n}\eq D_{\frac{3n+1}{3N}}\cup D_{\frac{4n+3}{4N}}.
\end{align*}
In particular, $c(N,n)=0$.
\end{lemma}
\begin{proof}
Write $\cl{R}=D_{\frac{3n+1}{3N}}\cup D_{\frac{4n+3}{4N}}$ as usual and indeed \mbox{$\cl{R}_{N,n}\supseteq\cl{R}$}. For the reverse inclusion, consider any $D_\frac{a}{b}$ with $N\divides b$, $\gcd(a,b)=1$, and $\frac{n}{N}\leq \frac{a}{b}\leq \frac{n+1}{N}$. We must have $b\geq 3N$ by the assumptions.

Case 1: $b=3N$. Then, $a$ can only be $3n+1$ by the second assumption, and hence $D_\frac{a}{b}\subseteq\cl{R}$.

Case 2: $b=4N$. Then, $a$ can only be $4n+1$ or $4n+3$. Note that $D\big(\frac{4n+1}{4N},\frac{1}{4N}\big)\subseteq D_\frac{3n+1}{3N}$ and $D\big(\frac{4n+3}{4N},\frac{1}{4N}\big)=D_\frac{4n+3}{4N}$, so $D_\frac{a}{b}\subseteq\cl{R}$.

Case 3: $b\geq 5N$. Then, $D_\frac{a}{b}$ is contained in the region
\begin{align*}
    \bigcup_{\frac{5n+1}{5N}\leq\alpha\leq\frac{5n+4}{5N}}D\big(\alpha,\tfrac{1}{5N}\big)
\end{align*}
where $\alpha$ ranges over all real numbers in the interval. Visually, this is the region swept by a closed disk from $D\big(\frac{5n+1}{5N},\frac{1}{5N}\big)$ to $D\big(\frac{5n+4}{5N},\frac{1}{5N}\big)$ (see the green region in \Cref{visual-proof-3n-4n}). It is an easy computation that the intersection points of the boundaries of $D_{\frac{3n+1}{3N}}$ and $D_{\frac{4n+3}{4N}}$ have $y$-coordinates equal $\pm\frac{1}{5N}$, so the swept region is contained in $\cl{R}$.

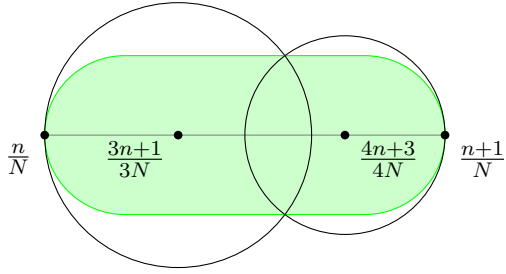
\begin{figure}[H]
\centering
\begin{tikzpicture}[scale=0.439]
\fill[fill=green!20!white,draw=green] (2.4,2.4) arc [start angle=90, end angle=270, radius=2.4cm];
\fill[fill=green!20!white] (2.39,-2.4) rectangle (9.61,2.4);
\fill[fill=green!20!white,draw=green] (9.6,-2.4) arc [start angle=-90, end angle=90, radius=2.4cm];
\draw[green](2.39,2.4)--(9.61,2.4);
\draw[green](2.39,-2.4)--(9.61,-2.4);
\draw[gray, thin](0,0)--(12,0);
\draw (4,0) circle [radius=4cm];
\draw (9,0) circle [radius=3cm];
\draw (0,0) node[below=9pt, left=1pt]{$\frac{n}{N}$};
\draw (4,0) node[below=9pt, left=1pt]{$\frac{3n+1}{3N}$};
\draw (9,0) node[below=9pt, right=1pt]{$\frac{4n+3}{4N}$};
\draw (12,0) node[below=9pt, right=1pt]{$\frac{n+1}{N}$};
\fill[black] (0,0) circle [radius=0.13cm];
\fill[black] (4,0) circle [radius=0.13cm];
\fill[black] (9,0) circle [radius=0.13cm];
\fill[black] (12,0) circle [radius=0.13cm];
\end{tikzpicture}
\caption{\label{visual-proof-3n-4n}Visualization of the proof}
\end{figure}
\end{proof}

We also write down the symmetric version of \Cref{fourpoint1}. The proof follows either from the same idea or directly from \Cref{fourpoint1} and \Cref{symmetric}.

\begin{lemma}
\label{fourpoint2}
Let $(N,n)$ be a bad pair. Suppose that
\begin{enumerate}[label=\rm{(\alph*)}]
    \item $\gcd(N,2n+1)>1$;
    \item $\gcd(N,3n+1)>1$ and $\gcd(N,3n+2)=1$;
    \item $\gcd(N,4n+1)=1$.
\end{enumerate}
Then,
\begin{align*}
    \cl{R}_{N,n}\eq D_{\frac{3n+2}{3N}}\cup D_{\frac{4n+1}{4N}}.
\end{align*}
In particular, $c(N,n)=0$.
\end{lemma}

\begin{corollary}
Let $N>1$ be an integer. If $\omega(N)=3$ and $p_1(N)=2$, then $c(N)=0$.
\end{corollary}
\begin{proof}
Write $N=2p_2(N)p_3(N)=2p_2p_3$ by the assumption. Fix a pair $(N,n)$. We may suppose that both of $n+1$ and $2n+1$ are not coprime to $N$ as otherwise $c(N,n)=0$ by \Cref{2coprime}. Since $N$ is even, replacing $n$ with $N-n-1$ by \Cref{symmetric}, we may assume that $n$ is even.

Now, both $n+1$ and $2n+1$ are odd, so $p_2$ and $p_3$ must each divide one of $n+1$ and $2n+1$. It is easy to check that $n+1$, $2n+1$, and $3n+1$ are pairwise coprime, so $p_2p_3\nmid 3n+1$. Also, $2\nmid 3n+1$ and $2\divides 3n+2$ as $2\divides n$. Hence, $\gcd(N,3n+1)=1$ and $\gcd(N,3n+2)>1$.  Similarly, it is easy to check that $n+1$, $2n+1$, and $4n+3$ are pairwise coprime, so $p_2p_3\nmid 4n+3$. Also, $2\nmid 4n+3$. Hence, $\gcd(N,4n+3)=1$.

In conclusion, all three assumptions in \Cref{fourpoint1} are satisfied, so $c(N,n)=0$. Hence, $c(N)=0$ by \Cref{badregion}.
\end{proof}

\section{The case when $\omega(N)=4$ and $p_1(N)\geq 3$}
\label{omega4-section}

In this section, we will discuss the final cases when $c(N)=0$. This section is divided into three cases, where each one is more subtle than the previous one. 

\subsection{$\omega(N)=4$ and $p_1(N)\geq 7$}
We will first prove that $c(N)=0$ when $\omega(N)=4$ and $p_1(N)\geq 7$. As we have previously investigated disks of radius $\frac{1}{4N}$, we now prove a corollary of \Cref{fourpoint1} and \Cref{fourpoint2}. 

\begin{corollary}
\label{4coprime}
For any pair $(N,n)$, $c(N,n)=0$ if three of 
\begin{align*}
    n+1,\;2n+1,\;3n+1,\;3n+2,\;4n+1,\;4n+3
\end{align*}
are coprime to $N$. 
\end{corollary}
\begin{proof}
Fix a pair $(N,n)$. If one of the following holds
\begin{enumerate}[label=\rm{(\arabic*)}]
    \item one of $n+1,2n+1$ is coprime to $N$;
    \item two of $n+1,2n+1,3n+1,3n+2$ are coprime to $N$,
\end{enumerate}
then $c(N,n)=0$ by \Cref{2coprime} or \Cref{3coprime}. Thus, suppose that both conditions do not hold. By the assumption, this implies that $(N,n)$ is a bad pair satisfying that
\begin{enumerate}[label=\rm{(\alph*)}]
    \item $\gcd(N,2n+1)>1$;
    \item exactly one of $3n+1,3n+2$ is coprime to $N$;
    \item $\gcd(N,4n+1)=\gcd(N,4n+3)=1$.
\end{enumerate}
Hence, $c(N,n)=0$ by \Cref{fourpoint1} and \Cref{fourpoint2}.
\end{proof}

\begin{proposition}
\label{propomega4}
Let $N>1$ be an integer. Then, $c(N)=0$ if for any integer $n$ with $0\leq n\leq N-1$, three of 
\begin{align*}
    n,\;n+1,\;2n+1,\;3n+1,\;3n+2,\;4n+1,\;4n+3
\end{align*}
are coprime to $N$. 
\end{proposition}
\begin{proof}
By \Cref{paircomp}, it suffices to show that $c(N,n)=0$ for all the pairs $(N,n)$. Now, for a fixed pair $(N,n)$, the assumption implies that three of $n+1,2n+1,3n+1,3n+2,4n+1,4n+3$ are coprime to $N$, so $c(N,n)=0$ by \Cref{3coprime}. The result then follows.
\end{proof}

As before, we have now obtained the first ingredient of the proof. The other ingredient of the proof on the common prime divisors is given by the following lemma similar to \Cref{pleq2}, which also explains the reason for the extra restriction $p_1(N)\geq 7$ in this case.

\begin{lemma}
\label{pleq5}
Let $p$ be a prime and $n$ be an integer such that there exist distinct
\begin{align*}
    a,b\in\{n,n+1,2n+1,3n+1,3n+2,4n+1,4n+3\}
\end{align*}
with $p\divides a$ and $p\divides b$. Then, $p\leq 5$.
\end{lemma}
\begin{proof}
Write $a=a_1n+a_2$ and $b=b_1n+b_2$. Then, $p\divides a$ and $p\divides b$ imply that $p\divides a_1b_2-a_2b_1$. Here, the only possible values of $|a_1b_2-a_2b_1|$ are $1,2,3,5,8$, so we must have $p\leq 5$.
\end{proof}

\begin{corollary}
Let $N>1$ be an integer. If $\omega(N)=4$ and $p_1(N)\geq 7$, then $c(N)=0$.
\end{corollary}
\begin{proof}
Fix an integer $n$ with $0\leq n\leq N-1$. As $p_i\geq p_1\geq 7$, by \Cref{pleq5}, each $p_i$ divides at most one of
\begin{align*}
    n,\;n+1,\;2n+1,\;3n+1,\;3n+2,\;4n+1,\;4n+3
\end{align*}
for every $0\leq n\leq N-1$.
As $\omega(N)=4$, there exist at least three of 
\begin{align*}
    n,\;n+1,\;2n+1,\;3n+1,\;3n+2,\;4n+1,\;4n+3
\end{align*}
that are coprime to $N$. Hence, $c(N)=0$ by \Cref{propomega4}.    
\end{proof}

\subsection{$\omega(N)=4$ and $p_1(N)=5$}
We will now prove that $c(N)=0$ if $\omega(N)=4$ and $p_1(N)=5$. Assume that $N>1$ with $\omega(N)=4$ and $p_1(N)=5$. Fix an integer $n$ with $0\leq n\leq N-1$.

\textbf{Case 1:} $n\equiv 0,2,4\pmod{5}$. Then, $5$ divides at most one of
\begin{align*}
    n,\;n+1,\;2n+1,\;3n+1,\;3n+2,\;4n+1,\;4n+3,
\end{align*}
and so does $p_i$ for $i\geq 2$ by \Cref{pleq5}. As $\omega(N)=4$, there exist at least three of 
\begin{align*}
    n,\;n+1,\;2n+1,\;3n+1,\;3n+2,\;4n+1,\;4n+3
\end{align*}
that are coprime to $N$. Hence, $c(N,n)=0$ by \Cref{4coprime}.

\textbf{Case 2:} $n\equiv 1,3\pmod{5}$. We may suppose that $(N,n)$ is a bad pair, i.e., $\gcd(N,n)>1$ and $\gcd(N,n+1)>1$. Replacing $n$ with $N-n-1$ by \Cref{symmetric}, we may assume that $n\equiv 1\pmod{5}$. Suppose that $\gcd(N,2n+1)>1$ as otherwise $c(N,n)=0$ by \Cref{midpoint}. Since $n\equiv1\pmod{5}$, $5$ does not divide any of $n,n+1,2n+1$. As $\omega(N)=4$, each $p_i$ for $i\geq 2$ must then divide a distinct element in $\{n,n+1,2n+1\}$, and hence each $p_i$ for $i\geq 2$ does not divide any of $3n+1,3n+2,4n+1,4n+3$ by \Cref{pleq5}. Now, $n\equiv 1\pmod{5}$ implies that $5\divides 3n+2$, $5\nmid 3n+1$, and $5\nmid 4n+3$. Hence, we obtain a bad pair $(N,n)$ satisfying that
\begin{enumerate}[label=(\alph*)]
    \item $\gcd(N,2n+1)>1$;
    \item $\gcd(N,3n+1)=1$ and $\gcd(N,3n+2)>1$;
    \item $\gcd(N,4n+3)=1$.
\end{enumerate}
By \Cref{fourpoint1}, $c(N,n)=0$.

In conclusion, $c(N)=0$ when $\omega(N)=4$ and $p_1(N)=5$.

\subsection{$\omega(N)=4$ and $p_1(N)=3$}
We will now prove that $c(N)=0$ if $\omega(N)=4$ and $p_1(N)=3$. This is the most subtle case as we are forced to analyze disks of radius $\frac{1}{5N}$.

Assume that $N>1$ with $\omega(N)=4$ and $p_1(N)=3$. Fix an integer $n$ with $0\leq n\leq N-1$.

\textbf{Case 1:} $n\equiv 1\pmod{3}$. Then, $3\divides 2n+1$ and $3$ does not divide any of
\begin{align*}
    n,\;n+1,\;3n+1,\;3n+2,\;4n+1,\;4n+3.
\end{align*}
If $p_2\geq 7$, then each $p_i$ for $i\geq 2$ can divide at most one of \begin{align*}
    n,\;n+1,\;3n+1,\;3n+2,\;4n+1,\;4n+3
\end{align*}
by \Cref{pleq5}, so at least three of
\begin{align*}
    n,\;n+1,\;3n+1,\;3n+2,\;4n+1,\;4n+3
\end{align*}
are coprime to $N$ as $\omega(N)=4$. Hence, $c(N,n)=0$ by \Cref{4coprime}.

Suppose that $p_2=5$. We may assume that $(N,n)$ is a bad pair so that $\gcd(N,n)>1$ and $\gcd(N,n+1)>1$. Consider first that $5$ divides one of $n,n+1$. Then, it is easy to check that $5$ cannot divide any of
\begin{align*}
    3n+1,\;3n+2,\;4n+1,\;4n+3.
\end{align*}
As $p_3,p_4\geq 7$, $p_3$ and $p_4$ can each divide at most one of
\begin{align*}
    n,\;n+1,\;3n+1,\;3n+2,\;4n+1,\;4n+3
\end{align*}
by \Cref{pleq5}, so at least three of
\begin{align*}
    n,\;n+1,\;3n+1,\;3n+2,\;4n+1,\;4n+3
\end{align*}
are coprime to $N$. Hence, $c(N,n)=0$ by \Cref{4coprime}.

Now suppose that $5$ does not divide any of $n,n+1$. Then, we must have $p_3\divides n$ and $p_4\divides n+1$ or the other way around, so both $p_3$ and $p_4$ do not divide any of
\begin{align*}
    3n+1,\;3n+2,\;4n+1,\;4n+3
\end{align*}
by \Cref{pleq5}. Here there are essentially three cases:
\begin{enumerate}[label=(\alph*)]
    \item $5$ does not divide any of $3n+1,3n+2,4n+1,4n+3$;
    \item $5\divides 3n+1$, $5\divides 4n+3$, $5\nmid 3n+2$, and $5\nmid 4n+1$;
    \item $5\divides 3n+2$, $5\divides 4n+1$, $5\nmid 3n+1$, and $5\nmid 4n+3$.
\end{enumerate}
It is easy to check that
\begin{align*}
    \cl{R}_{N,n}\eq
    \begin{cases}
        D_\frac{3n+1}{3N}\cup D_\frac{3n+2}{3N}&\text{ for (a) by \Cref{threepoint};} \\
        D_\frac{3n+2}{3N}\cup D_\frac{4n+1}{4N}&\text{ for (b) by \Cref{fourpoint2};} \\
        D_\frac{3n+1}{3N}\cup D_\frac{4n+3}{4N}&\text{ for (c) by \Cref{fourpoint1}.}
    \end{cases}
\end{align*}
Hence, $c(N,n)=0$.

Before proceeding to the remaining cases, we prove two lemmas on bad regions, where we will discuss disks of radius $\frac{1}{5N}$ as mentioned before.

\begin{lemma}
\label{35outs}
Let $(N,n)$ be a bad pair. Suppose that
\begin{enumerate}[label=\rm{(\alph*)}]
    \item each of $2n+1,3n+2,4n+3,5n+3$ is not coprime to $N$;
    \item $\gcd(N,3n+1)=\gcd(N,5n+4)=1$.
\end{enumerate}
Then, 
\begin{align*}
    \cl{R}_{N,n}\eq D_\frac{3n+1}{3N}\cup D_\frac{5n+4}{5N}.
\end{align*}
In particular, $c(N,n)=0$.
\end{lemma}
\begin{proof}
Let $\cl{R}=D_\frac{3n+1}{3N}\cup D_\frac{5n+4}{5N}$. It is clear that $\cl{R}\subseteq\cl{R}_{N,n}$ by the second assumption. For the reverse inclusion, consider any $D_\frac{a}{b}$ with $N\divides b$, $\gcd(a,b)=1$, and $\frac{n}{N}\leq \frac{a}{b}\leq \frac{n+1}{N}$. We must have $b\geq 3N$ by the first assumption.

Case 1: $b=3N$. Then, $a$ can only be $3n+1$ by the assumptions, so $D_\frac{a}{b}\subseteq D_\frac{3n+1}{3N}$.

Case 2: $b=4N$. Then, the only potential disk of radius $\frac{1}{4N}$ is $D\big(\frac{4n+1}{4N},\frac{1}{4N}\big)\subseteq D_\frac{3n+1}{3N}$, so $D_\frac{a}{b}\subseteq\cl{R}$.

Case 3: $b=5N$. Then, the potential disks of radius $\frac{1}{5N}$ are $D\big(\frac{5n+1}{5N},\frac{1}{5N}\big)\subseteq D_\frac{3n+1}{3N}$, $D\big(\frac{5n+2}{5N},\frac{1}{5N}\big)\subseteq D_\frac{3n+1}{3N}$, and $D\big(\frac{5n+4}{5N},\frac{1}{5N}\big)=D_\frac{5n+4}{5N}$, so $D_\frac{a}{b}\subseteq\cl{R}$.

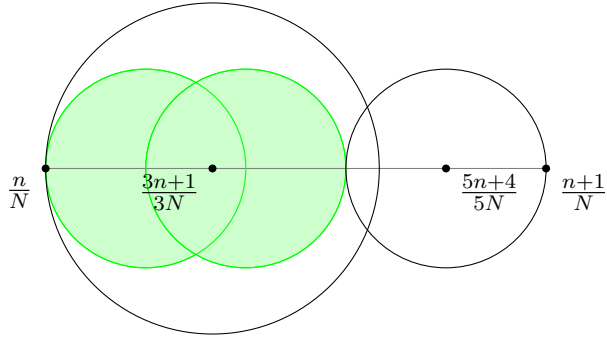
\begin{figure}[H]
\centering
\begin{tikzpicture}[scale=0.439]
\fill[fill=green!20!white,draw=green] (3,0) circle [radius=3cm];
\fill[fill=green!20!white,draw=green] (6,0) circle [radius=3cm];
\draw[gray, thin](0,0)--(15,0);
\draw[green] (3,0) circle [radius=3cm];
\draw[green] (6,0) circle [radius=3cm];
\draw (5,0) circle [radius=5cm];
\draw (12,0) circle [radius=3cm];
\draw (0,0) node[below=9pt, left=1pt]{$\frac{n}{N}$};
\draw (5,0) node[below=9pt, left=1pt]{$\frac{3n+1}{3N}$};
\draw (12,0) node[below=9pt, right=1pt]{$\frac{5n+4}{5N}$};
\draw (15,0) node[below=9pt, right=1pt]{$\frac{n+1}{N}$};
\fill[black] (0,0) circle [radius=0.12cm];
\fill[black] (5,0) circle [radius=0.12cm];
\fill[black] (12,0) circle [radius=0.12cm];
\fill[black] (15,0) circle [radius=0.12cm];
\end{tikzpicture}
\caption{Potential disks of radius $\frac{1}{5N}$}
\end{figure}

Case 4: $b=6N$. Then, the potential disks of radius $\frac{1}{6N}$ are $D\big(\frac{6n+1}{6N},\frac{1}{6N}\big)\subseteq D_\frac{3n+1}{3N}$ and $D\big(\frac{6n+5}{6N},\frac{1}{6N}\big)\subseteq D_\frac{5n+4}{5N}$, so $D_\frac{a}{b}\subseteq\cl{R}$.

\begin{figure}[H]
\centering
\begin{tikzpicture}[scale=0.439]
\fill[fill=green!20!white,draw=green] (2.5,0) circle [radius=2.5cm];
\fill[fill=green!20!white,draw=green] (12.5,0) circle [radius=2.5cm];
\draw[gray, thin](0,0)--(15,0);
\draw (5,0) circle [radius=5cm];
\draw (12,0) circle [radius=3cm];
\draw (0,0) node[below=9pt, left=1pt]{$\frac{n}{N}$};
\draw (5,0) node[below=9pt, left=1pt]{$\frac{3n+1}{3N}$};
\draw (12,0) node[below=9pt, right=1pt]{$\frac{5n+4}{5N}$};
\draw (15,0) node[below=9pt, right=1pt]{$\frac{n+1}{N}$};
\fill[black] (0,0) circle [radius=0.12cm];
\fill[black] (5,0) circle [radius=0.12cm];
\fill[black] (12,0) circle [radius=0.12cm];
\fill[black] (15,0) circle [radius=0.12cm];
\end{tikzpicture}
\caption{Potential disks of radius $\frac{1}{6N}$}
\end{figure}
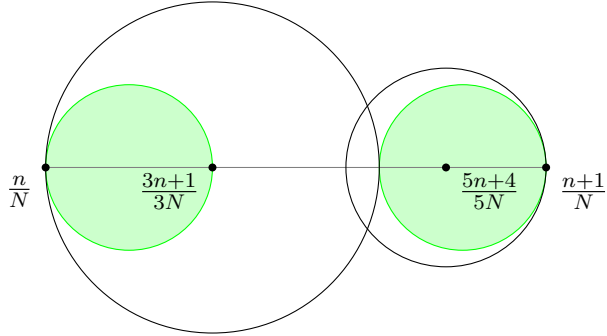

Case 5: $b=7N$. Then, the potential disks of radius $\frac{1}{7N}$ cover the following green region, and it can be shown by computations that the green region lies in $\cl{R}$.

\begin{figure}[H]
\centering
\begin{tikzpicture}[scale=0.439]
\fill[fill=green!20!white,draw=green] (2.143,0) circle [radius=2.143cm];
\fill[fill=green!20!white,draw=green] (4.286,0) circle [radius=2.143cm];
\fill[fill=green!20!white,draw=green] (6.429,0) circle [radius=2.143cm];
\fill[fill=green!20!white,draw=green] (8.571,0) circle [radius=2.143cm];
\fill[fill=green!20!white,draw=green] (10.714,0) circle [radius=2.143cm];
\fill[fill=green!20!white,draw=green] (12.857,0) circle [radius=2.143cm];
\draw[gray, thin](0,0)--(15,0);
\draw[green] (2.143,0) circle [radius=2.143cm];
\draw[green] (4.286,0) circle [radius=2.143cm];
\draw[green] (6.429,0) circle [radius=2.143cm];
\draw[green] (8.571,0) circle [radius=2.143cm];
\draw[green] (10.714,0) circle [radius=2.143cm];
\draw[green] (12.857,0) circle [radius=2.143cm];
\draw (5,0) circle [radius=5cm];
\draw (12,0) circle [radius=3cm];
\draw (0,0) node[below=9pt, left=1pt]{$\frac{n}{N}$};
\draw (5,0) node[below=9pt, left=1pt]{$\frac{3n+1}{3N}$};
\draw (12,0) node[below=9pt, right=1pt]{$\frac{5n+4}{5N}$};
\draw (15,0) node[below=9pt, right=1pt]{$\frac{n+1}{N}$};
\fill[black] (0,0) circle [radius=0.12cm];
\fill[black] (5,0) circle [radius=0.12cm];
\fill[black] (12,0) circle [radius=0.12cm];
\fill[black] (15,0) circle [radius=0.12cm];
\end{tikzpicture}
\caption{Potential disks of radius $\frac{1}{7N}$}
\end{figure}
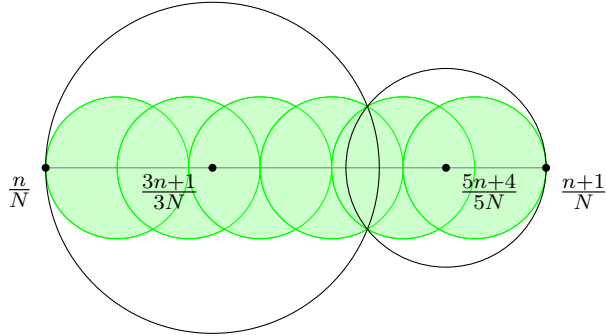

Case 6: $b=8N$. Then, the potential disks of radius $\frac{1}{8N}$ cover the following green region, and it can be shown by computations that the green region lies in $\cl{R}$.

\begin{figure}[H]
\centering
\begin{tikzpicture}[scale=0.439]
\fill[fill=green!20!white,draw=green] (1.875,0) circle [radius=1.875cm];
\fill[fill=green!20!white,draw=green] (5.625,0) circle [radius=1.875cm];
\fill[fill=green!20!white,draw=green] (9.375,0) circle [radius=1.875cm];
\fill[fill=green!20!white,draw=green] (13.125,0) circle [radius=1.875cm];
\draw[gray, thin](0,0)--(15,0);
\draw[green] (1.875,0) circle [radius=1.875cm];
\draw[green] (5.625,0) circle [radius=1.875cm];
\draw[green] (9.375,0) circle [radius=1.875cm];
\draw[green] (13.125,0) circle [radius=1.875cm];
\draw (5,0) circle [radius=5cm];
\draw (12,0) circle [radius=3cm];
\draw (0,0) node[below=9pt, left=1pt]{$\frac{n}{N}$};
\draw (5,0) node[below=9pt, left=1pt]{$\frac{3n+1}{3N}$};
\draw (12,0) node[below=9pt, right=1pt]{$\frac{5n+4}{5N}$};
\draw (15,0) node[below=9pt, right=1pt]{$\frac{n+1}{N}$};
\fill[black] (0,0) circle [radius=0.12cm];
\fill[black] (5,0) circle [radius=0.12cm];
\fill[black] (12,0) circle [radius=0.12cm];
\fill[black] (15,0) circle [radius=0.12cm];
\end{tikzpicture}
\caption{Potential disks of radius $\frac{1}{8N}$}
\end{figure}
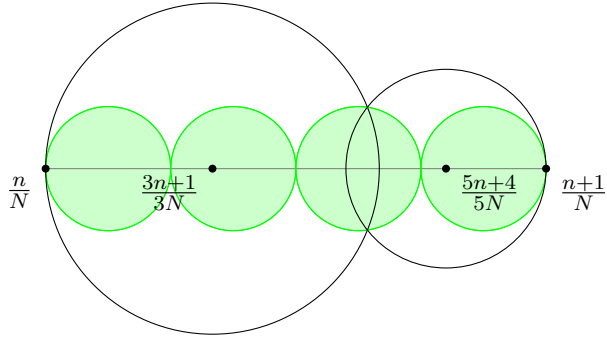

Case 7: $b\geq 9N$. All the potential disks of radius $\leq\frac{1}{9N}$ are contained in the region
\begin{align*}
    \bigcup_{\frac{9n+1}{9N}\leq\alpha\leq\frac{9n+8}{9N}}D\big(\alpha,\tfrac{1}{9N}\big)
\end{align*}
where $\alpha$ ranges over all real numbers in the interval. Visually, this is the region swept by a closed disk from $D\big(\frac{9n+1}{9N},\frac{1}{9N}\big)$ to $D\big(\frac{9n+8}{9N},\frac{1}{9N}\big)$ (see the green region in \Cref{disk-radius-9n}). It is an easy computation that the intersection points of the boundaries of $D_{\frac{3n+1}{3N}}$ and $D_{\frac{5n+4}{5N}}$ have $y$-coordinates equal to $\pm\frac{\sqrt{3}}{14N}$ and $\frac{1}{9N}<\frac{\sqrt{3}}{14N}$, so the swept region is contained in $\cl{R}$.

\begin{figure}[H]
\centering
\begin{tikzpicture}[scale=0.439]
\fill[fill=green!20!white,draw=green] (1.667,1.667) arc [start angle=90, end angle=270, radius=1.667cm];
\fill[fill=green!20!white] (1.665,-1.667) rectangle (13.335,1.667);
\fill[fill=green!20!white,draw=green] (13.333,-1.667) arc [start angle=-90, end angle=90, radius=1.667cm];
\draw[green](1.665,1.667)--(13.335,1.667);
\draw[green](1.665,-1.667)--(13.335,-1.667);
\draw[gray, thin](0,0)--(15,0);
\draw (5,0) circle [radius=5cm];
\draw (12,0) circle [radius=3cm];
\draw (0,0) node[below=9pt, left=1pt]{$\frac{n}{N}$};
\draw (5,0) node[below=9pt, left=1pt]{$\frac{3n+1}{3N}$};
\draw (12,0) node[below=9pt, right=1pt]{$\frac{5n+4}{5N}$};
\draw (15,0) node[below=9pt, right=1pt]{$\frac{n+1}{N}$};
\fill[black] (0,0) circle [radius=0.12cm];
\fill[black] (5,0) circle [radius=0.12cm];
\fill[black] (12,0) circle [radius=0.12cm];
\fill[black] (15,0) circle [radius=0.12cm];
\end{tikzpicture}
\caption{\label{disk-radius-9n}Potential disks of radius $\leq\frac{1}{9N}$}
\end{figure}
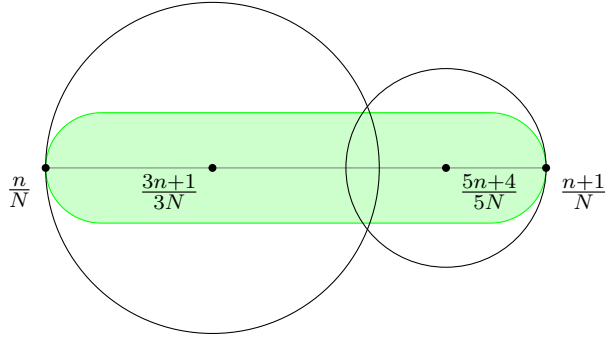

In conclusion, every potential disk $D_\frac{a}{b}$ lies in $\cl{R}$, so $\cl{R}_{N,n}=\cl{R}$, and hence $c(N,n)=0$.
\end{proof}

We also write down the symmetric version of \Cref{35outs}. The proof follows either from the same idea or directly from \Cref{35outs} and \Cref{symmetric}.

\begin{lemma}
\label{35outs2}
Let $(N,n)$ be a bad pair. Suppose that
\begin{enumerate}[label=\rm{(\alph*)}]
    \item each of $2n+1,3n+1,4n+1,5n+2$ is not coprime to $N$;
    \item $\gcd(N,3n+2)=\gcd(N,5n+1)=1$.
\end{enumerate}
Then, 
\begin{align*}
    \cl{R}_{N,n}\eq D_\frac{3n+2}{3N}\cup D_\frac{5n+1}{5N}.
\end{align*}
In particular, $c(N,n)=0$.
\end{lemma}

Now, we proceed back to the proof. We may suppose that $(N,n)$ is a bad pair, i.e., $\gcd(N,n)>1$ and $\gcd(N,n+1)>1$, that $\gcd(N,2n+1)>1$, and that one of $3n+1,3n+2$ is not coprime to $N$, as otherwise $c(N,n)=0$ by \Cref{goodregion}, \Cref{midpoint}, or \Cref{threepoint}.

\textbf{Case 2:} $n\equiv 0,2\pmod{3}$. Replacing $n$ with $N-n-1$ by \Cref{symmetric}, we may assume that $n\equiv 0\pmod{3}$. Then, $3\divides n$, $3\divides 4n+3$, and $3$ does not divide any of
\begin{align*}
    n+1,\;2n+1,\;3n+1,\;3n+2.
\end{align*}
Now, both $n+1$ and $2n+1$ are not coprime to $N$, and each $p_i$ for $i\geq 2$ divides at most one of 
\begin{align*}
    n+1,\;2n+1,\;3n+1,\;3n+2
\end{align*}
by \Cref{pleq2}. As $\omega(N)=4$, we obtain that exactly one of $3n+1,3n+2$ is coprime to $N$.

Suppose first that $\gcd(N,3n+1)=1$ so that $\gcd(N,3n+2)>1$. By the previous discussion, we must have that $p_2,p_3,p_4$ each divides a distinct element in $\{n+1,2n+1,3n+2\}$. It is easy to check that for $p\geq 5$, $p$ divides at most one of
\begin{align*}
    n+1,\;2n+1,\;3n+2,\;5n+4.
\end{align*}
This implies that $p_i\nmid 5n+4$ for $i\geq 2$. Also, $n\equiv0\pmod{3}$ implies that $3\divides 4n+3$, $3\divides 5n+3$, and $3\nmid 5n+4$. We thus obtain a bad pair $(N,n)$ satisfying that
\begin{enumerate}[label=\rm{(\alph*)}]
    \item each of $2n+1,3n+2,4n+3,5n+3$ is not coprime to $N$;
    \item $\gcd(N,3n+1)=\gcd(N,5n+4)=1$.
\end{enumerate}
Hence, $c(N,n)=0$ by \Cref{35outs}.

Now, suppose that $\gcd(N,3n+2)=1$ so that $\gcd(N,3n+1)>1$. Similarly, $p_2,p_3,p_4$ each divides a distinct element in $\{n+1,2n+1,3n+1\}$. It is easy to check that for $p\geq 5$, $p$ divides at most one of
\begin{align*}
    n+1,\;2n+1,\;3n+1,\;4n+1.
\end{align*}
This implies that $p_i\nmid 4n+1$ for $i\geq 2$. As $3\nmid 4n+1$, this implies that $\gcd(N,4n+1)=1$. We thus obtain a bad pair $(N,n)$ satisfying that
\begin{enumerate}[label=\rm{(\alph*)}]
    \item $\gcd(N,2n+1)>1$;
    \item $\gcd(N,3n+1)>1$ and $\gcd(N,3n+2)=1$;
    \item $\gcd(N,4n+1)=1$.
\end{enumerate}
Hence, $c(N,n)=0$ by \Cref{fourpoint2}.

In conclusion, $c(N)=0$ when $\omega(N)=4$ and $p_1(N)=3$.

\section{The cases when $c(N)\geq 1$}
\label{comp1-section}

In this section, we will show that $c(N)\geq 1$ in the remaining cases, and hence prove \Cref{mainthm}. Explicitly, we will prove the following proposition.

\begin{proposition}
\label{notouts}
Let $N>1$ be an integer. Then, $c(N)\geq 1$ if one of the following holds
\begin{enumerate}[label=\rm{(\arabic*)}]
    \item $\omega(N)\geq 5$;
    \item $\omega(N)\geq 4$ and $p_1(N)=2$.
\end{enumerate}
\end{proposition}

Note that by \Cref{paircomp}, $c(N)=0$ if and only if $c(N,n)=0$ for every pair $(N,n)$. Hence, the first step is to consider when a pair can have nonzero complexity, i.e., $c(N,n)\geq 1$. The motivation is the shape of the bad region that appears in (the symmetric version of) the proof of \Cref{fourpoint1}.

Indeed, for fixed $n$ and $N$, it is easy to show that the boundaries of the three disks $D\big(\frac{3n+2}{3N},\frac{1}{3N}\big)$, $D\big(\frac{4n+1}{4N},\frac{1}{4N}\big)$, and $D\big(\frac{5n+2}{5N},\frac{1}{5N}\big)$ intersect above the real line at exactly one point, which is precisely the north pole of $D\big(\frac{5n+2}{5N},\frac{1}{5N}\big)$ (see \Cref{intersect-3n-4n-5n}). Hence, a possible way of creating a pair $(N,n)$ with nonzero complexity is to have both $D\big(\frac{3n+2}{3N},\frac{1}{3N}\big)$ and $D\big(\frac{5n+2}{5N},\frac{1}{5N}\big)$ strictly lie on the boundary of $\cl{R}_{N,n}$, so that the north pole of $D\big(\frac{5n+2}{5N},\frac{1}{5N}\big)$ is an intersection point on the boundary.

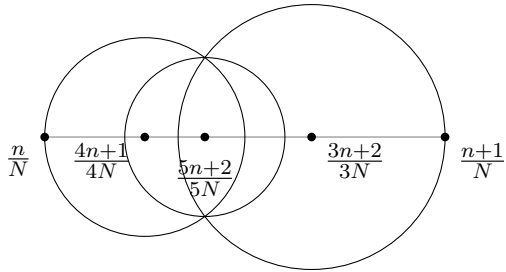
\begin{figure}[H]
\centering
\begin{tikzpicture}[scale=0.439]
\draw[gray, thin](0,0)--(12,0);
\draw (8,0) circle [radius=4cm];
\draw (4.8,0) circle [radius=2.4cm];
\draw (3,0) circle [radius=3cm];
\draw (0,0) node[below=9pt, left=1pt]{$\frac{n}{N}$};
\draw (8,0) node[below=9pt, right=1pt]{$\frac{3n+2}{3N}$};
\draw (4.8,0) node[below=5pt]{$\frac{5n+2}{5N}$};
\draw (3,0) node[below=9pt, left=1pt]{$\frac{4n+1}{4N}$};
\draw (12,0) node[below=9pt, right=1pt]{$\frac{n+1}{N}$};
\fill[black] (0,0) circle [radius=0.13cm];
\fill[black] (8,0) circle [radius=0.13cm];
\fill[black] (4.8,0) circle [radius=0.13cm];
\fill[black] (3,0) circle [radius=0.13cm];
\fill[black] (12,0) circle [radius=0.13cm];
\end{tikzpicture}
\caption{\label{intersect-3n-4n-5n}Intersection of $D\big(\frac{3n+2}{3N},\frac{1}{3N}\big)$, $D\big(\frac{4n+1}{4N},\frac{1}{4N}\big)$, and $D\big(\frac{5n+2}{5N},\frac{1}{5N}\big)$}
\end{figure}

To create such a situation, we first need to assume that $\gcd(N,3n+2)=\gcd(N,5n+2)=1$. We would also need to assume that each of 
\begin{align*}
    n,\;n+1,\;2n+1,\;3n+1,\;4n+1
\end{align*}
is not coprime to $N$, so that both of the disks $D\big(\frac{3n+2}{3N},\frac{1}{3N}\big)$ and $D\big(\frac{5n+2}{5N},\frac{1}{5N}\big)$ stay strictly on the boundary of $\cl{R}_{N,n}$. We will prove that these assumptions imply that $c(N,n)\geq 1$.

\begin{lemma}
\label{35region}
Let $(N,n)$ be a pair. Suppose that
\begin{enumerate}[label=\rm{(\alph*)}]
    \item each of $n,n+1,2n+1,3n+1,4n+1$ is not coprime to $N$;
    \item $\gcd(N,3n+2)=\gcd(N,5n+2)=1$.
\end{enumerate}
Then, $c(N,n)\geq 1$.
\end{lemma}
\begin{proof}
Indeed, $D_\frac{3n+2}{3N}\subseteq\cl{R}_{N,n}$ and $D_\frac{5n+2}{5N}\subseteq\cl{R}_{N,n}$ by the assumptions. By our discussion, it suffices to show that both $D_\frac{3n+2}{3N}$ and $D_\frac{5n+2}{5N}$ lie strictly on the boundary, i.e., $\frac{3n+2}{3N},\frac{5n+2}{5N}\in S_{N,n}$. We will prove this ``visually''. Note that $D_\frac{3n+2}{3N}$ is the only disk of radius $\geq\frac{1}{3N}$ that lies in $\cl{R}_{N,n}$, so we turn to considering disks of radius $\leq\frac{1}{4N}$.

The only potential disk of radius $\frac{1}{4N}$ is $D\big(\frac{4n+3}{4N},\frac{1}{4N}\big)\subseteq D_\frac{3n+2}{3N}$. The potential disks of radius $\frac{1}{5N}$ cover the following green region.
\begin{figure}[H]
\centering
\begin{tikzpicture}[scale=0.439]
\fill[fill=green!20!white,draw=green] (2.4,0) circle [radius=2.4cm];
\fill[fill=green!20!white,draw=green] (7.2,0) circle [radius=2.4cm];
\fill[fill=green!20!white,draw=green] (9.6,0) circle [radius=2.4cm];
\draw[gray, thin](0,0)--(12,0);
\draw[green] (2.4,0) circle [radius=2.4cm];
\draw (8,0) circle [radius=4cm];
\draw[green] (7.2,0) circle [radius=2.4cm];
\draw (4.8,0) circle [radius=2.4cm];
\draw[green] (9.6,0) circle [radius=2.4cm];
\draw (0,0) node[below=9pt, left=1pt]{$\frac{n}{N}$};
\draw (8,0) node[below=9pt, right=1pt]{$\frac{3n+2}{3N}$};
\draw (4.8,0) node[below=5pt]{$\frac{5n+2}{5N}$};
\draw (12,0) node[below=9pt, right=1pt]{$\frac{n+1}{N}$};
\fill[black] (0,0) circle [radius=0.13cm];
\fill[black] (8,0) circle [radius=0.13cm];
\fill[black] (4.8,0) circle [radius=0.13cm];
\fill[black] (12,0) circle [radius=0.13cm];
\end{tikzpicture}
\caption{Potential disks of radius $\frac{1}{5N}$}
\end{figure}
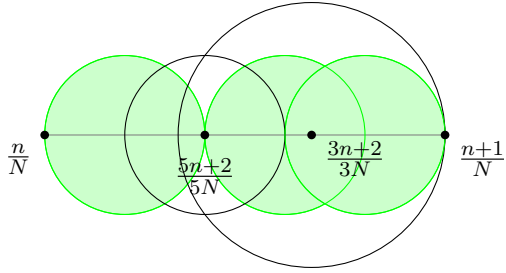

For potential disks of radius $\leq\frac{1}{6N}$, they all lie in the region
\begin{align*}
    \bigcup_{\frac{6n+1}{6N}\leq\alpha\leq\frac{6n+5}{6N}}D\big(\alpha,\frac{1}{6N}\big)
\end{align*}
where $\alpha$ ranges over all real numbers in the interval. Visually, this is the region swept by a closed disk from $D\big(\frac{6n+1}{6N},\frac{1}{6N}\big)$ to $D\big(\frac{6n+5}{6N},\frac{1}{6N}\big)$ (see the green region in \Cref{disk-radius-6n}).
\begin{figure}[H]
\centering
\begin{tikzpicture}[scale=0.439]
\fill[fill=green!20!white,draw=green] (2,2) arc [start angle=90, end angle=270, radius=2cm];
\fill[fill=green!20!white,draw=green] (10,-2) arc [start angle=-90, end angle=90, radius=2cm];
\fill[fill=green!20!white] (1.95,-2) rectangle (10.05,2);
\draw[gray, thin](0,0)--(12,0);
\draw[green] (1.95,2)--(10.05,2);
\draw[green] (1.95,-2)--(10.05,-2);
\draw (8,0) circle [radius=4cm];
\draw (4.8,0) circle [radius=2.4cm];
\draw (0,0) node[below=9pt, left=1pt]{$\frac{n}{N}$};
\draw (8,0) node[below=9pt, right=1pt]{$\frac{3n+2}{3N}$};
\draw (4.8,0) node[below=5pt]{$\frac{5n+2}{5N}$};
\draw (12,0) node[below=9pt, right=1pt]{$\frac{n+1}{N}$};
\fill[black] (0,0) circle [radius=0.13cm];
\fill[black] (4.8,0) circle [radius=0.13cm];
\fill[black] (8,0) circle [radius=0.13cm];
\fill[black] (12,0) circle [radius=0.13cm];
\end{tikzpicture}
\caption{\label{disk-radius-6n}Potential disks of radius $\leq\frac{1}{6N}$}
\end{figure}
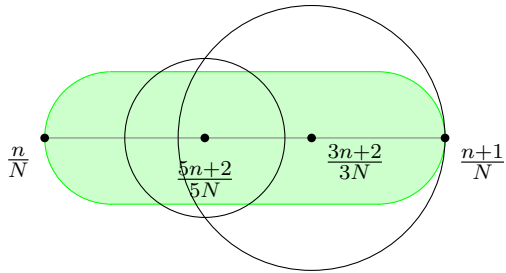

In conclusion, none of the remaining potential disks cover the intersection point of the boundaries of $D_\frac{3n+2}{3N}$ and $D_\frac{5n+2}{5N}$ completely, so both disks must lie strictly on the boundary of $\cl{R}_{N,n}$. Hence, $c(N,n)\geq c\big(P_{\frac{5n+2}{5N}}\big)\geq 1$.
\end{proof}

\begin{corollary}
\label{35n}
Let $N>1$ be an integer. Then, $c(N)\geq 1$ if there exists an integer $n$ with $0\leq n\leq N-1$ such that \begin{enumerate}[label=\rm{(\alph*)}]
    \item each of $n,n+1,2n+1,3n+1,4n+1$ is not coprime to $N$;
    \item $\gcd(N,3n+2)=\gcd(N,5n+2)=1$.
\end{enumerate}
\end{corollary}
\begin{proof}
If such $n$ exists, then $c(N,n)\geq 1$ by \Cref{35region}, so $c(N)\geq c(N,n)\geq 1$ by \Cref{badregion}.
\end{proof}

Before proving \Cref{notouts}, we will state one last lemma.

\begin{lemma}
\label{pleq5n}
Let $n$ be an integer and $p$ be a prime such that there exist
\begin{align*}
    a\in\{n,n+1,2n+1,3n+1,4n+1\}\text{ and }b\in\{3n+2,5n+2\}
\end{align*}
with $p\divides a$ and $p\divides b$. Then, $p\leq 5$.
\end{lemma}
\begin{proof}
Write $a=a_1n+a_2$ and $b=b_1n+b_2$. Then, $p\divides a$ and $p\divides b$ implies that $p\divides a_1b_2-a_2b_1$. Here, the only possible values of $|a_1b_2-a_2b_1|$ are $1,2,3,5$, so we must have $p\leq 5$.
\end{proof}

We are now ready to prove \Cref{notouts}. Assume first that $\omega(N)\geq 5$. As before, we write $p_i=p_i(N)$ for $1\leq i\leq \omega(N)$.

\textbf{Case 1:} $p_1\geq 7$. Let $0\leq a_i\leq p_i-1$ such that
\begin{align*}
    p_1\divides a_1, \; p_2\divides a_2+1, \; p_3\divides 2a_3+1, \; p_4\divides 3a_4+1,  \; p_5\divides 4a_5+1,
\end{align*}
and $p_i\divides a_i$ for $6\leq i\leq \omega(N)$. By the Chinese remainder theorem, there exists a unique integer~$n$ with $0\leq n\leq\prod_{i=1}^{\omega(N)}p_i-1\leq N-1$ such that
\begin{align*}
    n\;\equiv\;a_i\pmod{p_i}\quad\text{ for all }1\leq i\leq \omega(N).
\end{align*}
Now,
\begin{align*}
    p_1\divides n, \; p_2\divides n+1, \; p_3\divides 2n+1, \; p_4\divides 3n+1,  \; p_5\divides 4n+1,
\end{align*}
and $p_i\divides n$ for $6\leq i\leq \omega(N)$. As $p_i\geq p_1\geq 7$ for all $i$, by \Cref{pleq5n}, this implies that $p_i\nmid 3n+2$ and $p_i\nmid 5n+2$ for all $i$. In conclusion, we find an integer $n$ with $0\leq n\leq N-1$ such that
\begin{enumerate}[label=\rm{(\alph*)}]
    \item each of $n,n+1,2n+1,3n+1,4n+1$ is not coprime to $N$;
    \item $\gcd(N,3n+2)=\gcd(N,5n+2)=1$.
\end{enumerate}
Hence, $c(N)\geq 1$ by \Cref{35n}.

\textbf{Case 2:} $p_1=5$. We may construct $a_i$ and $n$ exactly as in Case 1. Note that in particular $5\divides n$ so $5\nmid 3n+2$ and $5\nmid 5n+2$. By the same reasoning as before, $n$ satisfies the conditions in \Cref{35n}, so $c(N)\geq 1$.

\textbf{Case 3:} $p_1=3$. We may still construct $a_i$ and $n$ exactly as in Case 1. Here one should check that $p_i\nmid 3n+2$ and $p_i\nmid 5n+2$ separately for $i=1$ and for $i=2$ if $p_2=5$, By the same reasoning as before, $n$ satisfies the conditions in \Cref{35n}, so $c(N)\geq 1$.
\\

Now, consider the last case when $p_1=2$ and $\omega(N)\geq 4$. Let $0\leq a_i\leq p_i-1$ such that $a_1=1$ (so that $2\divides a_1+1$ and $2\divides 3a_1+1$), 
\begin{align*}
    p_2\divides a_2,\; p_3\divides 2a_3+1, \; p_4\divides 4a_4+1, 
\end{align*}
and $p_i\divides a_i$ for $5\leq i\leq \omega(N)$. We may then construct $n$ as before using the Chinese remainder theorem. Here one should check that $p_i\nmid 3n+2$ and $p_i\nmid 5n+2$ separately for $i=1$, $i=2$ if $p_2=3$ or $5$, and $i=3$ if $p_3=5$. By the same reasoning as before, $n$ satisfies the conditions in \Cref{35n}, so $c(N)\geq 1$.

In conclusion, $c(N)\geq 1$ if $\omega(N)\geq 5$ or $\omega(N)\geq 4$ and \mbox{$p_1(N)=2$}, so \Cref{notouts} then follows.

From this, we also conclude \Cref{mainthm} by
\begin{enumerate}[label=(\arabic*)]
    \item \Cref{omega2-section} for the case when $\omega(N)\leq 2$,
    \item \Cref{omega3-section} for the case when $\omega(N)=3$,
    \item \Cref{omega4-section} for the case when $\omega(N)=4$ and $p_1(N)=3$, and
    \item \Cref{notouts} for the remaining cases.
\end{enumerate}

\section{Proof of \Cref{compinfty}}

In this section, we will prove \Cref{compinfty}. Throughout this section, let $q$ be a fixed prime number and $(N,n)$ be a fixed but arbitrary bad pair, i.e., an integer $N>1$ and an integer $0\leq n\leq N-1$ with $\gcd(N,n)>1$ and $\gcd(N,n+1)>1$.

Let $A=\lceil q^2/2\rceil\geq q^2/2$ and consider the family of disks
\begin{align*}
    D^{(j)}\deq D\bigg(\frac{(Aq+j)n+q}{(Aq+j)N}\thin,\thin\frac{1}{(Aq+j)N}\bigg)
\end{align*}
for $1\leq j\leq q-1$.

\begin{proposition}
\label{famdisk}
We have the following:
\begin{enumerate}[label=\rm{(\arabic*)}]
        \item for any $1\leq j\leq q-1$, there exists an arc of the boundary of $D^{(j)}$ of positive length that is not covered by $D^{(l)}$ for any $1\leq l\leq q-1$ with $l\neq j$;
    \item the north pole $P=\frac{(Aq+q-1)n+q}{(Aq+q-1)N}+\frac{1}{(Aq+q-1)N}\cdot i$ of $D^{(q-1)}$ lies in (the interior of) $D^{(j)}$ for all $1\leq j\leq q-2$.
\end{enumerate}
\end{proposition}
\begin{proof}
Note that scaling or shifting all $D^{(j)}$ simultaneously does not change their relative positions. Hence, we may replace each $D^{(j)}$ with
\begin{align*}
    D\bigg(\frac{q}{Aq+j}\thin,\thin\frac{1}{Aq+j}\bigg)\eq\big\{Nz-n\;\big|\; z\in D^{(j)}\big\}.
\end{align*}
Now, $P=\frac{q}{Aq+q-1}+\frac{1}{Aq+q-1}\cdot i$ is the north pole of $D^{(q-1)}$.

For (1), it suffices to show that for each $1\leq j\leq q-1$, there exists a point $Q_j=x_j+y_ji$ on the boundary of $D^{(j)}$ such that $Q_j\notin D^{(l)}$ for all $l\neq j$. Since $\C - D^{(l)}$ is open, the existence of such $Q_j$ leads to the existence of such arc. We will construct such $Q_j$ explicitly. Let
\begin{align}
\label{xj-yj-defn}
    x_j\eq\frac{q^2-1}{q(Aq+j)}\quad\text{ and }\quad y_j\eq\frac{\sqrt{q^2-1}}{q(Aq+j)}.
\end{align}
Indeed, it is easy to check that
\begin{align}
\label{xj-yj-circle-eqn}
    \bigg(x_j-\frac{q}{Aq+j}\bigg)^2+y_j^2\eq\bigg(\frac{1}{Aq+j}\bigg)^2,
\end{align}
so $Q_j=x_j+y_ji$ lies on the boundary of $D^{(j)}$. Now, let $l\neq j$ with $1\leq l\leq q-1$. Then,
\begin{align*}
    &\bigg(x_j-\frac{q}{Aq+l}\bigg)^2+y_j^2-\bigg(\frac{1}{Aq+l}\bigg)^2 \\
    \;\stackrel{(\ref{xj-yj-circle-eqn})}{=}\;&\bigg(\frac{2q}{Aq+j}-\frac{2q}{Aq+l}\bigg)\cdot x_j +\bigg(\frac{1-q^2}{(Aq+j)^2}-\frac{1-q^2}{(Aq+l)^2}\bigg) \\
    \;\stackrel{(\ref{xj-yj-defn})}{=}\;&\frac{2q(l-j)}{(Aq+j)(Aq+l)}\cdot \frac{q^2-1}{q(Aq+j)}+\frac{(1-q^2)(l-j)(2Aq+l+j)}{(Aq+j)^2(Aq+l)^2} \\
    \eq&\frac{(q^2-1)(l-j)}{(Aq+j)^2(Aq+l)^2}\big(2(Aq+l)-(2Aq+l+j)\big) \\
    \eq&\frac{(q^2-1)(l-j)^2}{(Aq+j)^2(Aq+l)^2}\;>\;0.
\end{align*}
Hence, $Q_j\notin D^{(l)}$ for all $l\neq j$. The result then follows.

For (2), write
\begin{align}
\label{x0-y0-defn}
    P\eq\frac{q}{Aq+q-1}+\frac{1}{Aq+q-1}\cdot i\;=:\;x_0+y_0i.
\end{align}
Note that
\begin{align}
\label{x0-y0-circle-eqn}
    \bigg(x_0-\frac{q}{Aq+q-1}\bigg)^2+y_0^2\eq\bigg(\frac{1}{Aq+q-1}\bigg)^2
\end{align}
since $P$ lies on the boundary of $D^{(q-1)}$. Now, fix $j$ with $1\leq j\leq q-2$. Then,
\begin{align*}
    &\bigg(x_0-\frac{q}{Aq+j}\bigg)^2+y_0^2-\bigg(\frac{1}{Aq+j}\bigg)^2 \\
    \;\stackrel{(\ref{x0-y0-circle-eqn})}{=}\;&\bigg(\frac{2q}{Aq+q-1}-\frac{2q}{Aq+j}\bigg)\cdot x_0 +\bigg(\frac{1-q^2}{(Aq+q-1)^2}-\frac{1-q^2}{(Aq+j)^2}\bigg) \\
    \;\stackrel{(\ref{x0-y0-defn})}{=}\;&\frac{2q(j-q+1)}{(Aq+q-1)(Aq+j)}\cdot\frac{q}{Aq+q-1}+\frac{(1-q^2)(j-q+1)(2Aq+j+q-1)}{(Aq+q-1)^2(Aq+j)^2} \\
    \eq&\frac{j-q+1}{(Aq+q-1)^2(Aq+j)^2}\big(2q^2(Aq+j)+(1-q^2)(2Aq+j+q-1)\big).
\end{align*}
Here, $j-q+1<0$, $(Aq+q-1)^2(Aq+j)^2>0$, and
\begin{align*}
    2q^2(Aq+j)+(1-q^2)(2Aq+j+q-1)\eq(2A-q^2)q+(j+1)q^2+(j-1)\;>\;0
\end{align*}
since $A\geq q^2/2$ and $1\leq j\leq q-2$. Hence, for $1\leq j\leq q-2$,
\begin{align*}
    \bigg(x_0-\frac{q}{Aq+j}\bigg)^2+y_0^2-\bigg(\frac{1}{Aq+j}\bigg)^2<\;0,
\end{align*}
so $P$ lies in the interior of $D^{(j)}$ for all $1\leq j\leq q-2$.
\end{proof}

\Cref{famdisk} essentially shows that we can obtain $q-1$ disks such that
\begin{enumerate}[label=(\arabic*)]
    \item each of them is not entirely covered by the union of the rest, and
    \item one of their north poles is covered by all other disks, i.e., one of the north poles has complexity equal to $q-2$.
\end{enumerate}
Hence, if we can force all $D^{(j)}$ to lie strictly on the boundary of $\cl{R}_N$ by putting extra assumptions on $N$, then we can conclude that $N$ with those extra assumptions must satisfy $c(N)\geq q-2$. To proceed with the idea, we need the following technical lemma.

\begin{lemma}
\label{pdivs1s2}
Fix an integer $n\geq 1$. Let $S_1,S_2$ be two finite nonempty subsets of
\begin{align*}
    \{(a,b)\in\Z^2\mid 0\leq b\leq a,\;\gcd(a,b)=1\}
\end{align*}
with $S_1\cap S_2=\emptyset$ and let
\begin{align*}
    C'\eq\max_{(a_i,b_i)\in S_i}|a_2b_1-a_1b_2|.
\end{align*}
Let
\begin{align*}
    T_i(n)\eq\{an+b\mid (a,b)\in S_i\}
\end{align*}
for $i=1,2$. Then, if a prime $p$ satisfies that $p\Z\cap T_i(n)\neq\emptyset$ for both $i=1,2$, then $p\leq C'$. 
\end{lemma}
\begin{remark}
This lemma can be viewed as a generalization of \Cref{pleq2}, \Cref{pleq5}, and \Cref{pleq5n}. It is important that here the constant $C'$ only depends on $S_1$ and $S_2$ but not on $n$.
\end{remark}
\begin{proof}
Let $p$ be a prime such that $p\Z\cap T_i(n)\neq\emptyset$ for both $i=1,2$. Then, there exists $(a_i,b_i)\in S_i$ such that $p\divides a_1n+b_1$ and $p\divides a_2n+b_2$. In particular,
\begin{align*}
    p\divides (a_1n+b_1)a_2-(a_2n+b_2)a_1=a_2b_1-a_1b_2.
\end{align*}
Hence, to prove that $p\leq C'=\max_{(a_i,b_i)\in S_i}|a_2b_1-a_1b_2|$, it suffices to show that $a_2b_1\neq a_1b_2$ for all $(a_i,b_i)\in S_i$.

Suppose that there exist $(a_i,b_i)\in S_i$ such that $a_2b_1=a_1b_2$. If $b_1=0$, then we must have $(a_1,b_1)=(a_2,b_2)=(1,0)$, contradicting that $S_1\cap S_2=\emptyset$. Hence, suppose that $b_1\neq 0$ and similarly $b_2\neq 0$. As $\gcd(a_1,b_1)=\gcd(a_2,b_2)=1$, we have $b_1|b_2$ and $b_2|b_1$. This implies that $b_1=b_2$ and hence $a_1=a_2$ as well, contradicting that $S_1\cap S_2=\emptyset$. The result then follows.
\end{proof}

Now, let
\begin{align*}
    A'\eq\bigg\lceil\frac{q(Aq+q-1)}{\sqrt{q^2-1}}\bigg\rceil.
\end{align*}
Let
\begin{align*}
    S_1\eq\{(Aq+j,q)\mid 1\leq j\leq q-1\}
\end{align*}
(note that $\gcd(Aq+j,q)=1$ for $1\leq j\leq q-1$) and
\begin{align*}
    S_2\eq\{(a,b)\in\Z^2\mid 0\leq b\leq a,\;\gcd(a,b)=1,\;a\leq A'\}\setminus S_1.
\end{align*}
For $i=1,2$, let
\begin{align*}
    T_i(n)\eq\{an+b\mid (a,b)\in S_i\}.
\end{align*}
Let
\begin{align*}
    B&\eq|S_2| \\
    C'&\eq\max_{(a_i,b_i)\in S_i}|a_2b_1-a_1b_2| \\
    C&\eq\max\{C', A'\}+1.
\end{align*}

\begin{lemma}
\label{twogcd}
Let $N>1$ be an integer. Suppose that $\omega(N)\geq B$ and $p_1(N)\geq C$. Then, there exists an integer $n$ with $0\leq n\leq N-1$ such that
\begin{enumerate}[label=\rm{(\arabic*)}]
    \item $\gcd(N,t_1)=1$ for all $t_1\in T_1(n)$;
    \item $\gcd(N,t_2)>1$ for all $t_2\in T_2(n)$.
\end{enumerate}
\end{lemma}
\begin{proof}
Write $p_i=p_i(N)$ as before. Let $\{(a^{(i)},b^{(i)})\}_{i=1}^B$ be an enumeration of $S_2$. For each $p_i$ with $1\leq i\leq B$, note that $p_i\geq C>A'\geq a^{(i)}>0$, so $p_i\nmid a^{(i)}$. Hence, there exists an integer $n_i$ with $0\leq n_i\leq p_i-1$ such that
\begin{align*}
    p_i\divides a^{(i)}n_i+b^{(i)}.
\end{align*}
For each $p_i$ with $i>B$, let $n_i=0$ so that $p_i\divides n_i$. By the Chinese remainder theorem, there exists an integer $n$ with $0\leq n\leq\prod_{i=1}^{\omega(N)}p_i-1\leq N-1$ such that
\begin{align*}
    p_i\divides a^{(i)}n+b^{(i)}\text{ for }1\leq i\leq B,\text{ and }p_i\divides n\text{ for }i>B.
\end{align*}
In particular, $(a^{(i)},b^{(i)})\in S_2$ and $(1,0)\in S_2$, so $p_i\Z\cap T_2(n)\neq\emptyset$ for all $1\leq i\leq \omega(N)$. Since $p_i\geq C>C'$, by \Cref{pdivs1s2}, we must have $p_i\Z\cap T_1(n)=\emptyset$, i.e., $p_i\nmid t_1$ for all $t_1\in T_1(n)$ and $1\leq i\leq \omega(N)$.

In conclusion, we obtain an integer $n$ with $0\leq n\leq N-1$ such that
\begin{enumerate}[label=(\arabic*)]
    \item for any $t_1\in T_1(n)$, $p_i\nmid t_1$ for all $1\leq i\leq \omega(N)$, and hence $\gcd(N,t_1)=1$, and
    \item for any $t_2\in T_2(n)$, $t_2=a^{(i)}n+b^{(i)}$ for some $1\leq i\leq B$, so $p_i\divides t_2$ by the construction, and hence $\gcd(N,t_2)>1$.
\end{enumerate}
The result then follows.
\end{proof}

\begin{proposition}
\label{compgeqc}
Let $(N,n)$ be a pair. Suppose that
\begin{enumerate}[label=\rm{(\arabic*)}]
    \item $\gcd(N,t_1)=1$ for all $t_1\in T_1(n)$, and
    \item $\gcd(N,t_2)>1$ for all $t_2\in T_2(n)$.
\end{enumerate}
Then, $c(N,n)\geq q-2$.
\end{proposition}
\begin{proof}
Note that $n,n+1\in T_2(n)$, so $\gcd(N,n)>1$ and $\gcd(N,n+1)>1$. Hence, $\cl{R}_{N,n}$ is a bad region. Since $\gcd(N,t_1)=1$ for all $t_1\in T_1(n)$, $D_{\frac{an+b}{aN}}$ lies in $\cl{R}_{N,n}$ for all $(a,b)\in S_1$, i.e., $D^{(j)}$ lies in $\cl{R}_{N,n}$ for all $1\leq j\leq q-1$; since $\gcd(N,t_2)>1$ for all $t_2\in T_2(n)$, any disk of the form $D\big(\frac{an+b}{aN},\frac{1}{aN}\big)$ for $(a,b)\in S_2$ will not appear in $\cl{R}_{N,n}$.
In particular, any other disks $D_{\frac{an+b}{aN}}$ that appear in $\cl{R}_{N,n}$ must have $a>A'$ and hence
\begin{align*}
    \frac{1}{a}\;<\;\frac{1}{A'}\;\leq\;\frac{\sqrt{q^2-1}}{q(Aq+q-1)}\;\leq\;\frac{\sqrt{q^2-1}}{q(Aq+j)}\eq y_j
\end{align*}
for all $1\leq j\leq q-1$, implying that $Q_j=x_j+y_ji\notin D_{\frac{an+b}{aN}}$. Here $Q_j$ is the point introduced in the proof of \Cref{famdisk} satisfying that $Q_j$ lies on the boundary of $D^{(j)}$ but $Q_j\notin D^{(l)}$ for any $l\neq j$. To summarize, for all $1\leq j\leq q-1$, 
\begin{enumerate}
    \item $Q_j$ lies on the boundary of $D^{(j)}$, and
    \item $Q_j\in D_{\frac{an+b}{aN}}$ for any disk $D_{\frac{an+b}{aN}}\subseteq\cl{R}_{N,n}$ if and only if $D_{\frac{an+b}{aN}}=D^{(j)}$.
\end{enumerate}
In particular, each $D^{(j)}$ for $1\leq j\leq q-1$ lies strictly on the boundary of $\cl{R}_{N,n}$. By \Cref{famdisk}, the north pole of $D^{(q-1)}=D_\frac{(Aq+q-1)n+q}{(Aq+q-1)N}$ is strictly covered by $D^{(j)}$ for all $1\leq j\leq q-2$. Hence, $c(N,n)\geq q-2$.
\end{proof}

\begin{corollary}
\Cref{compinfty} holds, i.e., if
\begin{align*}
    \omega(N)\;\geq\;\frac{q^6}{2} + 3q^4 - \frac{3q^3}{2} + \frac{9q^2}{2} - \frac{11q}{2} + 3
\end{align*}
and
\begin{align*}
    p_1(N)\;\geq\;\frac{q^6}{2} + 3q^4 - 2q^3 + \frac{9q^2}{2} - 7q + 3,
\end{align*}
then $c(N)\geq q-2$.
\end{corollary}
\begin{proof}
Note that
\begin{align*}
    A'\eq\bigg\lceil\frac{q(Aq+q-1)}{\sqrt{q^2-1}}\bigg\rceil&\;\leq\;\frac{\sqrt{q^2}}{\sqrt{q^2-1}}\big(\lceil q^2/2\rceil q+q-1\big)+1 \\
    &\;\leq\;2\big((q^2+1)q/2+q-1\big)+1 \\
    &\;\leq\;q^3+3q-1.
\end{align*}
Hence,
\begin{align*}
    B\eq|S_2|&\eq\hash\{(a,b)\in\Z^2\mid 0\leq b\leq a,\;\gcd(a,b)=1,\;a\leq A'\}-|S_1| \\
    &\;\leq\;1+\frac{(A'-1)A'}{2}-(q-1) \\
    &\;\leq\;\frac{q^6}{2} + 3q^4 - \frac{3q^3}{2} + \frac{9q^2}{2} - \frac{11q}{2} + 3.
\end{align*}
For $C'$, note that for all $(a_i,b_i)\in S_i$,
\begin{align*}
    a_2b_1-a_1b_2\;\leq\;A' q-(Aq+1)\cdot 0\;\leq\; q^4+3q^2-q
\end{align*}
and
\begin{align*}
    a_1b_2-a_2b_1\;\leq\;(Aq+q-1)(A'-1)-1\cdot q\;\leq\;\frac{q^6}{2} + 3q^4 - 2q^3 + \frac{9q^2}{2} - 7q + 2
\end{align*}
Hence,
\begin{align*}
    C'\;\leq\;\frac{q^6}{2} + 3q^4 - 2q^3 + \frac{9q^2}{2} - 7q + 2
\end{align*}
and
\begin{align*}
    C\eq\max\{C', A'\}+1\;\leq\;\frac{q^6}{2} + 3q^4 - 2q^3 + \frac{9q^2}{2} - 7q + 3.
\end{align*}

Now, 
\begin{align*}
    \omega(N)\;\geq\;\frac{q^6}{2} + 3q^4 - \frac{3q^3}{2} + \frac{9q^2}{2} - \frac{11q}{2} + 3\;\geq\; B
\end{align*}
and
\begin{align*}
    p_1(N)\;\geq\;\frac{q^6}{2} + 3q^4 - 2q^3 + \frac{9q^2}{2} - 7q + 3\;\geq\; C.
\end{align*}
The result then follows from \Cref{twogcd} and \Cref{compgeqc}.
\end{proof}

\bibliographystyle{amsalpha}
\bibliography{ref}

\providecommand{\bysame}{\leavevmode\hbox to3em{\hrulefill}\thinspace}
\providecommand{\MR}{\relax\ifhmode\unskip\space\fi MR }
% \MRhref is called by the amsart/book/proc definition of \MR.
\providecommand{\MRhref}[2]{%
  \href{http://www.ams.org/mathscinet-getitem?mr=#1}{#2}
}
\providecommand{\href}[2]{#2}
\begin{thebibliography}{AKU25}

\bibitem[AKU25]{abrams-katok-ugarcovici}
Adam Abrams, Svetlana Katok, and Ilie Ugarcovici, \emph{Reduction theory for
  {F}uchsian groups with cusps}, \url{https://arxiv.org/abs/2507.16958v2},
  2025.

\bibitem[Kat92]{katok-fuchsian}
Svetlana Katok, \emph{Fuchsian groups}, Chicago Lectures in Mathematics,
  University of Chicago Press, Chicago, IL, 1992. \MR{1177168}

\bibitem[KU10a]{katok-ugarcovici-structure-2010}
Svetlana Katok and Ilie Ugarcovici, \emph{Structure of attractors for
  {$(a,b)$}-continued fraction transformations}, J. Mod. Dyn. \textbf{4}
  (2010), no.~4, 637--691. \MR{2753948}

\bibitem[KU10b]{katok-ugarcovici-theory-2010}
\bysame, \emph{Theory of {$(a,b)$}-continued fraction transformations and
  applications}, Electron. Res. Announc. Math. Sci. \textbf{17} (2010), 20--33.
  \MR{2644834}

\bibitem[KU17]{katok-ugarcovici-2017}
\bysame, \emph{Structure of attractors for boundary maps associated to
  {F}uchsian groups}, Geom. Dedicata \textbf{191} (2017), 171--198.
  \MR{3941275}

\bibitem[Poh14]{pohl14}
Anke~D. Pohl, \emph{Symbolic dynamics for the geodesic flow on two-dimensional
  hyperbolic good orbifolds}, Discrete Contin. Dyn. Syst. \textbf{34} (2014),
  no.~5, 2173--2241. \MR{3124731}

\bibitem[PW26]{pohl-wabnitz}
Anke Pohl and Paul Wabnitz, \emph{Selberg {Z}eta {F}unctions, {C}uspidal
  {A}ccelerations, and {E}xistence of {S}trict {T}ransfer {O}perator
  {A}pproaches}, Mem. Amer. Math. Soc. \textbf{318} (2026), no.~1616, v+156.
  \MR{5042656}

\bibitem[Wab22]{wabnitz}
Paul Wabnitz, \emph{Strict transfer operator approaches for non-compact
  hyperbolic orbisurfaces}, \url{https://arxiv.org/abs/2209.06601v1}, 2022.

\bibitem[Zag23]{zagier}
Don Zagier, \emph{Reduction theory and periods of modular forms},
  \url{https://archive.mpim-bonn.mpg.de/id/eprint/4867/22/ManinSeminarTalk-fast.pdf},
  2023.

\end{thebibliography}

\end{document}